\documentclass[11pt, reqno]{amsart}
\usepackage{amssymb}
\usepackage{eucal}
\usepackage{amsmath}
\usepackage{amscd}
\usepackage[dvips]{color}
\usepackage{multicol}
\usepackage[all]{xy}           
\usepackage{graphicx}
\usepackage{color}
\usepackage{colordvi}
\usepackage{xspace}
\usepackage{tikz}
\usepackage{ifpdf}
\ifpdf
 \usepackage[colorlinks,final,backref=page,hyperindex]{hyperref}
\else
 \usepackage[colorlinks,final,backref=page,hyperindex,hypertex]{hyperref}
\fi

\begin{document}
\newcommand {\emptycomment}[1]{} 

\baselineskip=14pt
\newcommand{\nc}{\newcommand}
\newcommand{\delete}[1]{}
\nc{\mfootnote}[1]{\footnote{#1}} 
\nc{\todo}[1]{\tred{To do:} #1}

\delete{
\nc{\mlabel}[1]{\label{#1}}  
\nc{\mcite}[1]{\cite{#1}}  
\nc{\mref}[1]{\ref{#1}}  
\nc{\meqref}[1]{\ref{#1}} 
\nc{\mbibitem}[1]{\bibitem{#1}} 
}

\nc{\mlabel}[1]{\label{#1}  
{\hfill \hspace{1cm}{\bf{{\ }\hfill(#1)}}}}
\nc{\mcite}[1]{\cite{#1}{{\bf{{\ }(#1)}}}}  
\nc{\mref}[1]{\ref{#1}{{\bf{{\ }(#1)}}}}  
\nc{\meqref}[1]{\eqref{#1}{{\bf{{\ }(#1)}}}} 
\nc{\mbibitem}[1]{\bibitem[\bf #1]{#1}} 

\newtheorem{thm}{Theorem}[section]
\newtheorem{lem}[thm]{Lemma}
\newtheorem{cor}[thm]{Corollary}
\newtheorem{pro}[thm]{Proposition}
\theoremstyle{definition}
\newtheorem{defi}[thm]{Definition}
\newtheorem{ex}[thm]{Example}
\newtheorem{rmk}[thm]{Remark}
\newtheorem{pdef}[thm]{Proposition-Definition}
\newtheorem{condition}[thm]{Condition}

\renewcommand{\labelenumi}{{\rm(\alph{enumi})}}
\renewcommand{\theenumi}{\alph{enumi}}

\nc{\tred}[1]{\textcolor{red}{#1}}
\nc{\tblue}[1]{\textcolor{blue}{#1}}
\nc{\tgreen}[1]{\textcolor{green}{#1}}
\nc{\tpurple}[1]{\textcolor{purple}{#1}}
\nc{\btred}[1]{\textcolor{red}{\bf #1}}
\nc{\btblue}[1]{\textcolor{blue}{\bf #1}}
\nc{\btgreen}[1]{\textcolor{green}{\bf #1}}
\nc{\btpurple}[1]{\textcolor{purple}{\bf #1}}

\nc{\ld}[1]{\textcolor{purple}{Landry:#1}}
\nc{\hd}[1]{\textcolor{blue}{Damien:#1}}
\nc{\li}[1]{\textcolor{yellow}{#1}}
\nc{\lir}[1]{\textcolor{blue}{Li:#1}}


\nc{\twovec}[2]{\left(\begin{array}{c} #1 \\ #2\end{array} \right )}
\nc{\threevec}[3]{\left(\begin{array}{c} #1 \\ #2 \\ #3 \end{array}\right )}
\nc{\twomatrix}[4]{\left(\begin{array}{cc} #1 & #2\\ #3 & #4 \end{array} \right)}
\nc{\threematrix}[9]{{\left(\begin{matrix} #1 & #2 & #3\\ #4 & #5 & #6 \\ #7 & #8 & #9 \end{matrix} \right)}}
\nc{\twodet}[4]{\left|\begin{array}{cc} #1 & #2\\ #3 & #4 \end{array} \right|}

\nc{\rk}{\mathrm{r}}
\newcommand{\g}{\mathfrak g}
\newcommand{\h}{\mathfrak h}
\newcommand{\pf}{\noindent{$Proof$.}\ }
\newcommand{\frkg}{\mathfrak g}
\newcommand{\frkh}{\mathfrak h}
\newcommand{\Id}{\rm{Id}}
\newcommand{\gl}{\mathfrak {gl}}
\newcommand{\ad}{\mathrm{ad}}
\newcommand{\add}{\frka\frkd}
\newcommand{\frka}{\mathfrak a}
\newcommand{\frkb}{\mathfrak b}
\newcommand{\frkc}{\mathfrak c}
\newcommand{\frkd}{\mathfrak d}
\newcommand {\comment}[1]{{\marginpar{*}\scriptsize\textbf{Comments:} #1}}

\nc{\tforall}{\text{ for all }}

\nc{\svec}[2]{{\tiny\left(\begin{matrix}#1\\
#2\end{matrix}\right)\,}}  
\nc{\ssvec}[2]{{\tiny\left(\begin{matrix}#1\\
#2\end{matrix}\right)\,}} 

\nc{\typeI}{local cocycle $3$-Lie bialgebra\xspace}
\nc{\typeIs}{local cocycle $3$-Lie bialgebras\xspace}
\nc{\typeII}{double construction $3$-Lie bialgebra\xspace}
\nc{\typeIIs}{double construction $3$-Lie bialgebras\xspace}

\nc{\bia}{{$\mathcal{P}$-bimodule ${\bf k}$-algebra}\xspace}
\nc{\bias}{{$\mathcal{P}$-bimodule ${\bf k}$-algebras}\xspace}

\nc{\rmi}{{\mathrm{I}}}
\nc{\rmii}{{\mathrm{II}}}
\nc{\rmiii}{{\mathrm{III}}}
\nc{\pr}{{\mathrm{pr}}}
\newcommand{\huaA}{\mathcal{A}}

\nc{\OT}{constant $\theta$-}
\nc{\T}{$\theta$-}
\nc{\IT}{inverse $\theta$-}

\nc{\pll}{\beta}
\nc{\plc}{\epsilon}

\nc{\ass}{{\mathit{Ass}}}
\nc{\lie}{{\mathit{Lie}}}
\nc{\comm}{{\mathit{Comm}}}
\nc{\dend}{{\mathit{Dend}}}
\nc{\zinb}{{\mathit{Zinb}}}
\nc{\tdend}{{\mathit{TDend}}}
\nc{\prelie}{{\mathit{preLie}}}
\nc{\postlie}{{\mathit{PostLie}}}
\nc{\quado}{{\mathit{Quad}}}
\nc{\octo}{{\mathit{Octo}}}
\nc{\ldend}{{\mathit{ldend}}}
\nc{\lquad}{{\mathit{LQuad}}}

 \nc{\adec}{\check{;}} \nc{\aop}{\alpha}
\nc{\dftimes}{\widetilde{\otimes}} \nc{\dfl}{\succ} \nc{\dfr}{\prec}
\nc{\dfc}{\circ} \nc{\dfb}{\bullet} \nc{\dft}{\star}
\nc{\dfcf}{{\mathbf k}} \nc{\apr}{\ast} \nc{\spr}{\cdot}
\nc{\twopr}{\circ} \nc{\tspr}{\star} \nc{\sempr}{\ast}
\nc{\disp}[1]{\displaystyle{#1}}
\nc{\bin}[2]{ (_{\stackrel{\scs{#1}}{\scs{#2}}})}  
\nc{\binc}[2]{ \left (\!\! \begin{array}{c} \scs{#1}\\
    \scs{#2} \end{array}\!\! \right )}  
\nc{\bincc}[2]{  \left ( {\scs{#1} \atop
    \vspace{-.5cm}\scs{#2}} \right )}  
\nc{\sarray}[2]{\begin{array}{c}#1 \vspace{.1cm}\\ \hline
    \vspace{-.35cm} \\ #2 \end{array}}
\nc{\bs}{\bar{S}} \nc{\dcup}{\stackrel{\bullet}{\cup}}
\nc{\dbigcup}{\stackrel{\bullet}{\bigcup}} \nc{\etree}{\big |}
\nc{\la}{\longrightarrow} \nc{\fe}{\'{e}} \nc{\rar}{\rightarrow}
\nc{\dar}{\downarrow} \nc{\dap}[1]{\downarrow
\rlap{$\scriptstyle{#1}$}} \nc{\uap}[1]{\uparrow
\rlap{$\scriptstyle{#1}$}} \nc{\defeq}{\stackrel{\rm def}{=}}
\nc{\dis}[1]{\displaystyle{#1}} \nc{\dotcup}{\,
\displaystyle{\bigcup^\bullet}\ } \nc{\sdotcup}{\tiny{
\displaystyle{\bigcup^\bullet}\ }} \nc{\hcm}{\ \hat{,}\ }
\nc{\hcirc}{\hat{\circ}} \nc{\hts}{\hat{\shpr}}
\nc{\lts}{\stackrel{\leftarrow}{\shpr}}
\nc{\rts}{\stackrel{\rightarrow}{\shpr}} \nc{\lleft}{[}
\nc{\lright}{]} \nc{\uni}[1]{\tilde{#1}} \nc{\wor}[1]{\check{#1}}
\nc{\free}[1]{\bar{#1}} \nc{\den}[1]{\check{#1}} \nc{\lrpa}{\wr}
\nc{\curlyl}{\left \{ \begin{array}{c} {} \\ {} \end{array}
    \right .  \!\!\!\!\!\!\!}
\nc{\curlyr}{ \!\!\!\!\!\!\!
    \left . \begin{array}{c} {} \\ {} \end{array}
    \right \} }
\nc{\leaf}{\ell}       
\nc{\longmid}{\left | \begin{array}{c} {} \\ {} \end{array}
    \right . \!\!\!\!\!\!\!}
\nc{\ot}{\otimes} \nc{\sot}{{\scriptstyle{\ot}}}
\nc{\otm}{\overline{\ot}}
\nc{\ora}[1]{\stackrel{#1}{\rar}}
\nc{\ola}[1]{\stackrel{#1}{\la}}
\nc{\pltree}{\calt^\pl}
\nc{\epltree}{\calt^{\pl,\NC}}
\nc{\rbpltree}{\calt^r}
\nc{\scs}[1]{\scriptstyle{#1}} \nc{\mrm}[1]{{\rm #1}}
\nc{\dirlim}{\displaystyle{\lim_{\longrightarrow}}\,}
\nc{\invlim}{\displaystyle{\lim_{\longleftarrow}}\,}
\nc{\mvp}{\vspace{0.5cm}} \nc{\svp}{\vspace{2cm}}
\nc{\vp}{\vspace{8cm}} \nc{\proofbegin}{\noindent{\bf Proof: }}
\nc{\proofend}{$\blacksquare$ \vspace{0.5cm}}
\nc{\freerbpl}{{F^{\mathrm RBPL}}}
\nc{\sha}{{\mbox{\cyr X}}}  
\nc{\ncsha}{{\mbox{\cyr X}^{\mathrm NC}}} \nc{\ncshao}{{\mbox{\cyr
X}^{\mathrm NC,\,0}}}
\nc{\shpr}{\diamond}    
\nc{\shprm}{\overline{\diamond}}    
\nc{\shpro}{\diamond^0}    
\nc{\shprr}{\diamond^r}     
\nc{\shpra}{\overline{\diamond}^r}
\nc{\shpru}{\check{\diamond}} \nc{\catpr}{\diamond_l}
\nc{\rcatpr}{\diamond_r} \nc{\lapr}{\diamond_a}
\nc{\sqcupm}{\ot}
\nc{\lepr}{\diamond_e} \nc{\vep}{\varepsilon} \nc{\labs}{\mid\!}
\nc{\rabs}{\!\mid} \nc{\hsha}{\widehat{\sha}}
\nc{\lsha}{\stackrel{\leftarrow}{\sha}}
\nc{\rsha}{\stackrel{\rightarrow}{\sha}} \nc{\lc}{\lfloor}
\nc{\rc}{\rfloor}
\nc{\tpr}{\sqcup}
\nc{\nctpr}{\vee}
\nc{\plpr}{\star}
\nc{\rbplpr}{\bar{\plpr}}
\nc{\sqmon}[1]{\langle #1\rangle}
\nc{\forest}{\calf}
\nc{\altx}{\Lambda_X} \nc{\vecT}{\vec{T}} \nc{\onetree}{\bullet}
\nc{\Ao}{\check{A}}
\nc{\seta}{\underline{\Ao}}
\nc{\deltaa}{\overline{\delta}}
\nc{\trho}{\tilde{\rho}}

\nc{\rpr}{\circ}
\nc{\dpr}{{\tiny\diamond}}
\nc{\rprpm}{{\rpr}}

\nc{\mmbox}[1]{\mbox{\ #1\ }} \nc{\ann}{\mrm{ann}}
\nc{\Aut}{\mrm{Aut}} \nc{\can}{\mrm{can}}
\nc{\twoalg}{{two-sided algebra}\xspace}
\nc{\colim}{\mrm{colim}}
\nc{\Cont}{\mrm{Cont}} \nc{\rchar}{\mrm{char}}
\nc{\cok}{\mrm{coker}} \nc{\dtf}{{R-{\rm tf}}} \nc{\dtor}{{R-{\rm
tor}}}
\renewcommand{\det}{\mrm{det}}
\nc{\depth}{{\mrm d}}
\nc{\Div}{{\mrm Div}} \nc{\End}{\mrm{End}} \nc{\Ext}{\mrm{Ext}}
\nc{\Fil}{\mrm{Fil}} \nc{\Frob}{\mrm{Frob}} \nc{\Gal}{\mrm{Gal}}
\nc{\GL}{\mrm{GL}} \nc{\Hom}{\mrm{Hom}} \nc{\hsr}{\mrm{H}}
\nc{\hpol}{\mrm{HP}} \nc{\id}{\mrm{id}} \nc{\im}{\mrm{im}}
\nc{\incl}{\mrm{incl}} \nc{\length}{\mrm{length}}
\nc{\LR}{\mrm{LR}} \nc{\mchar}{\rm char} \nc{\NC}{\mrm{NC}}
\nc{\mpart}{\mrm{part}} \nc{\pl}{\mrm{PL}}
\nc{\ql}{{\QQ_\ell}} \nc{\qp}{{\QQ_p}}
\nc{\rank}{\mrm{rank}} \nc{\rba}{\rm{RBA }} \nc{\rbas}{\rm{RBAs }}
\nc{\rbpl}{\mrm{RBPL}}
\nc{\rbw}{\rm{RBW }} \nc{\rbws}{\rm{RBWs }} \nc{\rcot}{\mrm{cot}}
\nc{\rest}{\rm{controlled}\xspace}
\nc{\rdef}{\mrm{def}} \nc{\rdiv}{{\rm div}} \nc{\rtf}{{\rm tf}}
\nc{\rtor}{{\rm tor}} \nc{\res}{\mrm{res}} \nc{\SL}{\mrm{SL}}
\nc{\Spec}{\mrm{Spec}} \nc{\tor}{\mrm{tor}} \nc{\Tr}{\mrm{Tr}}
\nc{\mtr}{\mrm{sk}}

\nc{\ab}{\mathbf{Ab}} \nc{\Alg}{\mathbf{Alg}}
\nc{\Algo}{\mathbf{Alg}^0} \nc{\Bax}{\mathbf{Bax}}
\nc{\Baxo}{\mathbf{Bax}^0} \nc{\RB}{\mathbf{RB}}
\nc{\RBo}{\mathbf{RB}^0} \nc{\BRB}{\mathbf{RB}}
\nc{\Dend}{\mathbf{DD}} \nc{\bfk}{{\bf k}} \nc{\bfone}{{\bf 1}}
\nc{\base}[1]{{a_{#1}}} \nc{\detail}{\marginpar{\bf More detail}
    \noindent{\bf Need more detail!}
    \svp}
\nc{\Diff}{\mathbf{Diff}} \nc{\gap}{\marginpar{\bf
Incomplete}\noindent{\bf Incomplete!!}
    \svp}
\nc{\FMod}{\mathbf{FMod}} \nc{\mset}{\mathbf{MSet}}
\nc{\rb}{\mathrm{RB}} \nc{\Int}{\mathbf{Int}}
\nc{\Mon}{\mathbf{Mon}}
\nc{\remarks}{\noindent{\bf Remarks: }}
\nc{\OS}{\mathbf{OS}} 
\nc{\Rep}{\mathbf{Rep}}
\nc{\Rings}{\mathbf{Rings}} \nc{\Sets}{\mathbf{Sets}}
\nc{\DT}{\mathbf{DT}}

\nc{\BA}{{\mathbb A}} \nc{\CC}{{\mathbb C}} \nc{\DD}{{\mathbb D}}
\nc{\EE}{{\mathbb E}} \nc{\FF}{{\mathbb F}} \nc{\GG}{{\mathbb G}}
\nc{\HH}{{\mathbb H}} \nc{\LL}{{\mathbb L}} \nc{\NN}{{\mathbb N}}
\nc{\QQ}{{\mathbb Q}} \nc{\RR}{{\mathbb R}} \nc{\BS}{{\mathbb{S}}} \nc{\TT}{{\mathbb T}}
\nc{\VV}{{\mathbb V}} \nc{\ZZ}{{\mathbb Z}}


\nc{\calao}{{\mathcal A}} \nc{\cala}{{\mathcal A}}
\nc{\calc}{{\mathcal C}} \nc{\cald}{{\mathcal D}}
\nc{\cale}{{\mathcal E}} \nc{\calf}{{\mathcal F}}
\nc{\calfr}{{{\mathcal F}^{\,r}}} \nc{\calfo}{{\mathcal F}^0}
\nc{\calfro}{{\mathcal F}^{\,r,0}} \nc{\oF}{\overline{F}}
\nc{\calg}{{\mathcal G}} \nc{\calh}{{\mathcal H}}
\nc{\cali}{{\mathcal I}} \nc{\calj}{{\mathcal J}}
\nc{\call}{{\mathcal L}} \nc{\calm}{{\mathcal M}}
\nc{\caln}{{\mathcal N}} \nc{\calo}{{\mathcal O}}
\nc{\calp}{{\mathcal P}} \nc{\calq}{{\mathcal Q}} \nc{\calr}{{\mathcal R}}
\nc{\calt}{{\mathcal T}} \nc{\caltr}{{\mathcal T}^{\,r}}
\nc{\calu}{{\mathcal U}} \nc{\calv}{{\mathcal V}}
\nc{\calw}{{\mathcal W}} \nc{\calx}{{\mathcal X}}
\nc{\CA}{\mathcal{A}}

\nc{\fraka}{{\mathfrak a}} \nc{\frakB}{{\mathfrak B}}
\nc{\frakb}{{\mathfrak b}} \nc{\frakd}{{\mathfrak d}}
\nc{\oD}{\overline{D}}
\nc{\frakF}{{\mathfrak F}} \nc{\frakg}{{\mathfrak g}}
\nc{\frakm}{{\mathfrak m}} \nc{\frakM}{{\mathfrak M}}
\nc{\frakMo}{{\mathfrak M}^0} \nc{\frakp}{{\mathfrak p}}
\nc{\frakS}{{\mathfrak S}} \nc{\frakSo}{{\mathfrak S}^0}
\nc{\fraks}{{\mathfrak s}} \nc{\os}{\overline{\fraks}}
\nc{\frakT}{{\mathfrak T}}
\nc{\oT}{\overline{T}}
\nc{\frakX}{{\mathfrak X}} \nc{\frakXo}{{\mathfrak X}^0}
\nc{\frakx}{{\mathbf x}}
\nc{\frakTx}{\frakT}      
\nc{\frakTa}{\frakT^a}        
\nc{\frakTxo}{\frakTx^0}   
\nc{\caltao}{\calt^{a,0}}   
\nc{\ox}{\overline{\frakx}} \nc{\fraky}{{\mathfrak y}}
\nc{\frakz}{{\mathfrak z}} \nc{\oX}{\overline{X}}

\font\cyr=wncyr10

\nc{\al}{\alpha}
\nc{\lam}{\lambda}
\nc{\lr}{\longrightarrow}
\newcommand{\K}{\mathbb {K}}
\newcommand{\A}{\rm A}


\title[Relative (pre-)anti-flexible algebras and 
associated algebraic structures]{Relative
(pre-)anti-flexible algebras and 
associated algebraic structures}

\author[Mafoya Landry Dassoundo]{Mafoya Landry Dassoundo
}
\address[]{Chern Institute of Mathematics
\& LPMC, Nankai University, Tianjin 300071, China}
\email{
dassoundo@yahoo.com}

\date{\today}

\begin{abstract}
Pre-anti-flexible 
family algebras are introduced  and linked 
with the  notions of 
relative anti-flexible algebras, 
left and right pre-Lie family algebras
and relative Lie algebras which are for
mostly newly defined.
Relative pre-anti-flexible algebras are 
given and their underlying 
algebras structures such as 
pre-anti-flexible family algebras,
left and right pre-Lie family algebras, 
and other are investigated and significant identities 
linking those introduced structures are derived.
In addition, a generalization of the Rota-Baxter operators 
defined on a relative anti-flexible algebra  is 
introduced and both Rota-Baxter operators
and its generalization are used to build
relative pre-anti-flexible algebras structures underlying 
relative anti-flexible algebras
and related consequences are derived.
\end{abstract}

\subjclass[2010]{17A30; 16W99; 05E16}

\keywords{anti-flexible family algebra,
pre-anti-flexible family algebra,
relative anti-flexible algebra,
relative  pre-anti-flexible algebra,
relative Lie algebra,
Rota-Baxter family operators}

\maketitle

\tableofcontents

\numberwithin{equation}{section}

\tableofcontents

\numberwithin{equation}{section}

\allowdisplaybreaks

\section{Introduction and preliminaries}
The most best known, 
well geometrically interpreted  and 
deeply investigated 
algebras classes are 
associative algebras, Lie algebras and 
Poisson algebras. 
On one hand, it is worth to  notice 
that these  algebraic classes 
 (associative, Lie and Poisson) 
are not sufficient to describe, build  and-or characterize
all other algebraic classes.
On the other hand, it is advisable to establish links
among the most best known algebras
in the literature above cited, to other
 algebraic classes such as 
left-symmetric algebras (\cite{Vinberg}), 
anti-flexible algebras (\cite{Rodabaugh, Dassoundo_B_H}), 
dendriform algebras (\cite{Loday}),... etc 
in order to hold them otherwise for deep investigations.

Rota-Baxter operators, since their 
introductions (\cite{Baxter, Rota1, Rota2}) 
continue to reveal 
they usefulness in the  
process of the
construction of  new algebraic structures, of the 
splitting algebraic structures 
and many  other use.
More precisely, defending a 
Rota-Baxter operator
$R_B: A\rightarrow A$  on an associative 
algebra  $(A, \cdot)$,
then according to \cite{Aguiar1}, 
there is  a dendriform algebra 
structure $\prec, \succ: A\times A\rightarrow A$
underlying $A$  given by, for any 
$x,y\in A$, 
\begin{eqnarray}\label{eq:preantiflexible_under_rota}
x\succ y:= R_B(x)\cdot y,\;\;\;\;\;
x\prec y:=x\cdot R_B(y),
\end{eqnarray}
i.e. the linear products
$\prec, \succ:A\times A\rightarrow A$
above defined satisfying 
the following relations, 
for any $x,y, z\in A$,
\begin{subequations}
\begin{eqnarray}
(x\succ y)\prec z-x\succ (y\prec z)=0,
\end{eqnarray}
\begin{eqnarray}
(x\succ y+x\prec y)\succ z-x\succ(y\succ z)=0,
\end{eqnarray}
\begin{eqnarray}
x\prec(y\succ z+y\prec z)-(x\prec y)\prec z=0,
\end{eqnarray}
\end{subequations}
 if and only if the linear map $R_B: A\rightarrow A$
is a zero weight 
Rota-Baxter operator, that is $R_B$ satisfying for any $x,y\in A,$
\begin{eqnarray}\label{eq:zero_weight_RB}
R_B(x)\cdot R_B(y)=R_B(R_B(x)\cdot y)+
R_B(x\cdot R_B(y)).
\end{eqnarray}
Similarly, it is well known from \cite{Dassoundo} that, 
given a Rota-Baxter operator 
(zero weight Rota-Baxter operator)
defined on anti-flexible algebra $(A, \cdot)$, there is 
a pre-anti-flexible algebra structure 
$\prec, \succ: A\times A\rightarrow A$
underlying the vector space $A$, that is for any 
$x,y, z\in A$ the following relations are satisfied,
\begin{subequations}
\begin{eqnarray}\label{eq:pre-anti-flexible-rel1}
(x\succ y)\prec z-x\succ (y\prec z)=
(z\succ y)\prec x-z\succ (y\prec x),
\end{eqnarray}
\begin{eqnarray}\label{eq:pre-anti-flexible-rel2}
(x\succ y+x\prec y)\succ z-x\succ(y\succ z)=
z\prec(y\succ x+y\prec x)-(z\prec y)\prec x.
\end{eqnarray}
\end{subequations}
In addition, above result is extended i.e. 
 given a generalized Rota-Baxter operator
defined on and anti-flexible algebra $(A, \cdot)$ 
which is a linear map
$G_{RB}: A\rightarrow A$ satisfying for any 
$x,y,z\in A,$
\begin{eqnarray}
&&(G_{RB}(G_{RB}(x)\cdot y+x\cdot G_{RB}(y))
-G_{RB}(x)\cdot G_{RB}(y))\cdot z+\cr 
&&z\cdot(G_{RB}(y)\cdot G_{RB}(x)-
G_{RB}(G_{RB}(y)\cdot x+y\cdot G_{RB}(x) ))=0,
\end{eqnarray}
there is pre-anti-flexible algebra structure
$\prec, \succ: A\times A\rightarrow A$ defined on $A$ by
Eq.~\eqref{eq:preantiflexible_under_rota}
for which $R_B$ is $G_{RB}.$
Besides, dendriform and (di-)tri-algebras were 
introduced and related to Rota-Baxter operators and 
associated consequences were derived (\cite{Ebrahimi-Fard}).
Moreover, it is well known  
(from \cite{Gubarev_Kolesnikov}) that 
Koszul duality  of operad governing 
(di-)trialgebras is corresponding to 
 operad governing  variety of 
(di-)tridendriform algebras which are embedded 
to zero's weight Rota-Baxter 
algebra.
Furthermore, it is 
opportune to notice that 
 a general operadic definition for 
the notion of splitting 
algebraic structures is equivalent with
some Manin products of operads
which are closely related to Rota-Baxter
operators (\cite{Bai_Bellier_Guo_Ni}).
More generally,  splitting 
algebraic operations procedure in any algebraic operad theory were 
uniformed and linked to the introduced notion of 
Rota-Baxter operators on 
operads (\cite{Pei_Bai_Guo}) and
many other  results built from Rota-Baxter 
algebras were surveyed in \cite{Guo1} 
and the references therein.

The notion of operated semi-group was introduced 
to build some algebraic 
structures on combinatoric elements   
mainly the binary rooted trees. The most relevant examples 
are the construction of free Rota-Baxter algebras in terms of 
Motzkin paths and the planar rooted trees (\cite{Guo})
and the  used of the typed decorated trees theory for 
describing  the combinatorial
species (\cite{Bergeron_L_L}).
Given a (non)associative $\K$ (field of characteristic zero) algebra $A$,  
a Rota-Baxter family operators of weight $\lambda\; (\lambda\in \K)$ 
is the family of linear maps $P_{\omega}:A \rightarrow A$, 
where $\omega\in \Omega$ which is an associative semi-group,
and satisfying for any $x,y\in A$ and for any 
$\alpha, \beta \in \Omega$, 
\begin{eqnarray}\label{eq:Rota-Bater-Family-Algebra-weiged}
P_{\alpha}(x)P_{\beta}(y)=
P_{\alpha\beta}(xP_{\beta}(y))+
P_{\alpha\beta}(P_{\alpha}(x)y)+
\lambda P_{\alpha\beta}(xy).
\end{eqnarray}
Rota-Baxter family operators theory takes its origins in
renormalization theory of quantum 
field theory (\cite[page 591]{Fard_Bondia_Patras}). 
Recently, free (non)commutative Rota-Baxter family
was introduced and linked to (tri)dendriform 
family algebras (\cite{Zhang_Gao}).
Moreover,  Rota-Baxter family algebras indexed by associative 
semi-group were introduced and 
shown that they amount to an ordinary 
Rota-Baxter algebras structures on the tensor product
with associative semi-group algebras. Similar results 
were established with (tri)dendriform family algebras
and finally free dendriform family algebras were built 
in terms of typed decorated planar
binary trees and generalized 
 on free tridendriform 
family algebras (\cite{Zhang_Gao_Manchon}), and more generally, 
the notion of $\Omega$-dendriform structures
were introduced and it also was proved that 
nonassociative structures on 
typed binary trees were unify 
and generalized (\cite{Foissy}).
Similarly, pre-Lie family algebras and 
their freeness, which are their underlying 
(tri)dendriform family algebras, were introduced and 
related typed decorated trees were built and related 
generalization were also derived~(\cite{Manchon_Zhang}). 

Throughout this article, we notice that
$\Omega$ is an associative semi-group and 
$\Omega_c$ is a commutative associative semi-group, 
any considered algebra is defined as a finite dimensional 
over a field of characteristic zero and this
 despite the fact that some of results hold 
independently of the dimension of the considered 
algebra on which they were established.
We will end this introductory section by describing 
the content  flowchart in substance 
of this paper as follows.
In section~\ref{section2}, we introduce
the notion of pre-anti-flexible family algebras,
establish their relations with dendriform
family algebras and built its underlying
relative anti-flexible algebras as well as
relative Lie algebras, left and right
pre-Lie family algebras and 
related consequences are derived.
In section~\ref{section3}, 
relative pre-anti-flexible algebras
are introduced and viewed as a generalization 
of the relative dendriform algebras and their 
underlying relative pre-anti-flexible algebras, 
relative pre-Lie and right pre-Lie algebras, 
relative Lie algebras and other structures 
are derived. In addition, we show that there is 
 pre-anti-flexible family algebra, 
left and right pre-Lie family algebras underlying
a given relative pre-Lie algebra.
In section~\ref{section4}, we prove that the zero's
weight Rota-Baxter family operators defined on 
a relative  anti-flexible algebra and its generalization 
induce a relative pre-anti-flexible algebra structure.
It is moreover proved,
under some assumptions on relative 
anti-flexible algebra  that a Rota-Baxter family operator
defined on the underlying relative Lie algebra of 
the relative anti-flexible algebra  
induces relative pre-anti-flexible algebra structure.
\section{Pre-anti-flexible family  
algebras and related family algebras}\label{section2}
In this section, pre-anti-flexible family algebras 
are introduce and related consequences are derived.
Associated family algebraic structures
and relative algebras structure
are built in detail. 
\begin{defi}
A pre-anti-flexible family algebra is the quadruple
$(A, \prec_{\omega}, \succ_{\omega}, \Omega_c)$ such that
 $A$ is a vector space equipped with 
two operations families
$\prec_{\alpha}, \succ_{\alpha}: A\times A\rightarrow A$ 
for each $\alpha \in \Omega_c$
and satisfying 
for any $x,y,z \in A$, for all $\alpha, \beta\in \Omega_c$,
\begin{subequations}
\begin{eqnarray}\label{eq:preanti1}
(x\succ_{\alpha}y)\prec_{\beta}z-x\succ_{\alpha}(y\prec_{\beta} z)=
(z\succ_{\beta}y)\prec_{\alpha}x-z\succ_{\beta}(y\prec_{\alpha} x),
\end{eqnarray}
\begin{eqnarray}\label{eq:preanti2}
(x\succ_{\alpha}y+x\prec_{\beta} y)\succ_{\alpha\beta}z -
x\succ_{\alpha} (y \succ_{\beta}z)=
(z\prec_{\beta}y)\prec_{\alpha}x-
z\prec_{\beta\alpha}(y\succ_{\beta}x+y\prec_{\alpha}x).
\end{eqnarray}
\end{subequations}
\end{defi}

\begin{rmk}
If the LHS and the RHS of each of Eq.~\eqref{eq:preanti1} and 
Eq.~\eqref{eq:preanti2} are zero, 
then  pre-anti-flexible family algebra  is
dendriform family algebra 
(\cite{Zhang_Gao, Manchon_Zhang, Zhang_Gao_Manchon}).
Thus, pre-anti-flexible family algebra
can be considered as a generalization of dendriform family algebra.
\end{rmk}
\begin{defi}
 An $\Omega_c$-relative anti-flexible  algebra is a vector space $A$
 equipped with the operations families
 $\cdot_{_{\alpha,\beta}}:A\times A\rightarrow A$
 for each couple $(\alpha, \beta) \in \Omega_c^2$ 
 and satisfying, for any $x,y,z\in A$, 
  for any $\alpha, \beta, \gamma\in\Omega_c$,
 \begin{eqnarray}\label{eq:identityantiflexible}
 (x\cdot_{_{\alpha, \beta}} y)\cdot_{_{\alpha\beta, \gamma}}z
 +z\cdot_{_{\gamma, \beta\alpha}}(y\cdot_{_{\beta, \alpha}}x)-
 (z\cdot_{_{\gamma, \beta}} y)\cdot_{_{\gamma\beta, \alpha}} x-
 x\cdot_{_{\alpha, \beta\gamma}}(y\cdot_{_{\beta, \gamma}}z)=0,
 \end{eqnarray}
equivalently
\begin{eqnarray}\label{eq:associatorid}
(x,y, z)_{_{\alpha, \beta, \gamma}}=(z,y, x)_{_{\gamma, \beta, \alpha}},
\end{eqnarray}
where, for any $x,y,z\in A$ and for any $\alpha, \beta, \gamma\in \Omega_c$,
\begin{eqnarray}\label{eq:associator}
(x,y, z)_{_{\alpha, \beta, \gamma}}:=
 (x\cdot_{_{\alpha, \beta}} y)\cdot_{_{\alpha\beta, \gamma}}z-
 x\cdot_{_{\alpha, \beta\gamma}}(y\cdot_{_{\beta, \gamma}}z).
\end{eqnarray}
In the following of this paper, an $\Omega_c$-relative 
anti-flexible  algebra   above  defined will simply denote by
$(A, \cdot_{_{\omega_1, \omega_2}}, \Omega_c)$ or by $A$ if 
there is no other consideration.
\end{defi}

\begin{thm}\label{thm:underlingrelaticeantiflexible}
Let $(A, \prec_{\omega}, \succ_{\omega}, \Omega_c)$
be a pre-anti-flexible family algebra,
defining for all $\alpha, \beta\in \Omega_c$ and for any $x,y\in A$, 
\begin{eqnarray}\label{eq:underlyinganti}
x\ast_{_{\alpha, \beta}}y= x\succ_{\alpha} y+ x\prec_{\beta} y
\end{eqnarray}
determines  an $\Omega_c$-relative anti-flexible algebra
structure on $A$.   
\end{thm}
\begin{proof}
Let $x,y,z \in A$ and $\alpha, \beta, \gamma\in \Omega_c$. We have 
\begin{eqnarray*}
(x,y,z)_{_{\alpha, \beta, \gamma}}&=&
(x\ast_{_{\alpha, \beta}} y)\ast_{_{\alpha\beta, \gamma}}z-
 x\ast_{_{\alpha, \beta\gamma}}(y\ast_{_{\beta, \gamma}}z)\cr
 &=&
 (x\succ_{\alpha} y+ x\prec_{\beta} y)\succ_{_{\alpha\beta}}z+
 (x\succ_{\alpha} y+ x\prec_{\beta} y)\prec_{{\gamma}}z\cr
 &-&
 x\succ_{\alpha}(y\succ_{\beta} z+y\prec_{\gamma} z)
 -x\prec_{_{\beta\gamma}}(y\succ_{\beta} z+y\prec_{\gamma} z)\cr
 &=&
\{(x\succ_{\alpha} y+ x\prec_{\beta} y)\succ_{_{\alpha\beta}}z
- x\succ_{\alpha}(y\succ_{\beta} z\}\cr
&+&
\{
(x\prec_{\beta} y)\prec_{{\gamma}}z
-x\prec_{_{\beta\gamma}}(y\succ_{\beta} z+y\prec_{\gamma} z)
\}\cr&+& 
\{(x\succ_{\alpha} y)\prec_{{\gamma}}z
 -x\succ_{\alpha}(y\prec_{\gamma} z)\}\cr 
 &=&
\{(z\prec_{\beta}y)\prec_{\alpha}x
-z\prec_{\beta\alpha}(y\succ_{\beta}x+y\prec_{\alpha}x)\}\cr
&+&
\{(z\succ_{\gamma} y+z\prec_{\beta}y)\succ_{_{\gamma\beta}}x
-z\succ_{\gamma}(y\succ_{\beta}x)
\}\cr&+& 
\{(z\succ_{\gamma} y)\prec_{{\alpha}}x
 -z\succ_{\gamma}(y\prec_{\alpha} x)\}\cr 
 &=&
 (z\succ_{\gamma} y+z\prec_{\beta}y)\succ_{_{\gamma\beta}}x+
 (z\succ_{\gamma} y+z\prec_{\beta}y)\prec_{\alpha}x\cr 
 &-&
z\succ_{\gamma}(y\succ_{\beta}x+y\prec_{\alpha} x)
-z\prec_{\beta\alpha}(y\succ_{\beta}x+y\prec_{\alpha}x)\cr
 (x,y,z)_{_{\alpha, \beta, \gamma}}&=&
 (z, y, x)_{_{\gamma, \beta, \alpha}}.
\end{eqnarray*}
Therefore, the linear product defined by Eq.~\eqref{eq:underlyinganti}
satisfies Eq.~\eqref{eq:associatorid}, thus it
confers to $A$  an $\Omega_c$-relative anti-flexible  algebra structure.
\end{proof}
\begin{thm}
Let $A$ be a $\mathbf{k}$ vector space,
$\Omega_c$ is a commutative associative semi-group.
Consider the linear products 
$\prec, \succ: A\otimes\mathbf{k}\Omega_c\times A\otimes\mathbf{k}\Omega_c
\rightarrow A\otimes\mathbf{k}\Omega_c$ defined on 
$A\otimes\mathbf{k}\Omega_c$. The triple 
$(A\otimes\mathbf{k}\Omega_c, \prec, \succ)$ 
is a pre-anti-flexible algebra if and only if 
$(A, \prec_{\omega}, \succ_{\omega})$ is a
pre-anti-flexible family algebra, where for any 
$x,y,\in A$ and for any $\alpha, \beta\in \Omega_c$, 
\begin{subequations}
\begin{eqnarray}\label{eq:product1}
(x\otimes \alpha)\prec (y\otimes \beta):=
(x\prec_{\beta} y)\otimes \alpha\beta, 
\end{eqnarray}
\begin{eqnarray}\label{eq:product2}
(x\otimes \alpha)\succ (y\otimes \beta):=
(x\succ_{\alpha}y)\otimes \alpha\beta.
\end{eqnarray}
\end{subequations}
\end{thm}
\begin{proof}
Let $x,y,z\in A$ and $\alpha, \beta, \gamma \in \Omega_c.$
We have
\begin{equation*}
((x\otimes \alpha)\succ (y\otimes \beta))\prec (z\otimes \gamma)-
(x\otimes\alpha)\succ ((y\otimes \beta)\prec (z\otimes \gamma))=
((x\succ_{\alpha}y)\prec_{\gamma}z-
x\succ_{\alpha}(y\prec_{\gamma}z))\otimes\alpha\beta\gamma,
\end{equation*}
\begin{eqnarray*}
((x\otimes \alpha)\succ(y\otimes \beta)+
(x\otimes \alpha)\prec(y\otimes \beta))\succ (z\otimes \gamma)-
(x\otimes\alpha)\succ ((y\otimes \beta)\succ(z\otimes\gamma))\cr
=
(
(x\succ_{\alpha}y+
x\prec_{\beta}y)\succ_{\alpha\beta}z-
x\succ_{\alpha}(y\succ_{\beta}z)
)\otimes \alpha\beta\gamma,
\end{eqnarray*}
According  the commutativity of $\Omega_c$, we have  
\begin{equation*}
((z\otimes \gamma)\succ (y\otimes \beta))\prec (x\otimes \alpha)-
(z\otimes\gamma)\succ ((y\otimes \beta)\prec (x\otimes \alpha))=
((z\succ_{\gamma}y)\prec_{\alpha}x-
z\succ_{\gamma}(y\prec_{\alpha}x))\otimes\alpha\beta\gamma,
\end{equation*}
and finally 
\begin{eqnarray*}
((z\otimes  \gamma)\prec(y\otimes\beta))\prec (x\otimes\alpha)-
(z\otimes \gamma)\prec ((y\otimes \beta)\succ(x\otimes\alpha)
+(y\otimes\beta)\prec(x\otimes\alpha) )\\=
(
(z\prec_{\beta}y)\prec_{\alpha}x-
z\prec_{\beta\alpha}(y\succ_{\beta}x+y\prec_{\alpha}x)
)\otimes \alpha\beta\gamma.
\end{eqnarray*}
Hence, if  $(A\otimes\mathbf{k}\Omega_c, \prec, \succ)$ is 
a pre-anti-flexible algebra then
$(A, \prec_{\omega}, \succ_{\omega})$ is a
pre-anti-flexible family algebra too.

Conversely, if $(A, \prec_{\omega}, \succ_{\omega})$ is a
pre-anti-flexible family algebra such that its  product 
is derived by Eq.~\eqref{eq:product1} and Eq.~\eqref{eq:product2},
then $(A\otimes\mathbf{k}\Omega_c, \prec, \succ)$ is 
a pre-anti-flexible algebra.
\end{proof}
\begin{defi}(\cite{Manchon_Zhang})
A left pre-Lie family algebra is a vector space $A$ 
together with binary operations  
$\triangleright_{\omega}:A\times A\rightarrow A$
 for $\omega\in \Omega_c$,
 such that for any 
$x,y,z\in A$ and for any $\alpha, \beta \in \Omega_c$,
\begin{eqnarray}\label{eq:leftpreLieid}
(x\triangleright_{\alpha}y)\triangleright_{\alpha\beta}z-
x\triangleright_{\alpha}(y\triangleright_{\beta} z)=
(y\triangleright_{\beta}x)\triangleright_{\beta\alpha}z-
y\triangleright_{\beta}(x\triangleright_{\alpha}z).
\end{eqnarray}
\end{defi}
\begin{defi}
A right pre-Lie family algebra is a vector space $A$ 
equipped with binary operations 
$\triangleleft_{\omega}:A\times A\rightarrow A$ 
for $\omega\in \Omega_c$, such that for any 
$x,y,z\in A$ and for any $\alpha, \beta \in \Omega_c$,
\begin{eqnarray}\label{eq:rightpreLieid}
x\triangleleft_{\alpha\beta}(y\triangleleft_{\beta} z)-
(x\triangleleft_{\alpha}y)\triangleleft_{\beta}z=
x\triangleleft_{\beta\alpha}(z\triangleleft_{\alpha}y)-
(x\triangleleft_{\beta}z)\triangleleft_{\alpha}y.
\end{eqnarray}
\end{defi}

\begin{thm}
Let $(A, \prec_{\omega}, \succ_{\omega}, \Omega_c)$ be 
a pre-anti-flexible family algebra, defining the binary 
operations by, for all $x,y\in A$,
\begin{subequations}
\begin{eqnarray}\label{eq:prodleftpreLie}
x\triangleright_{\omega} y:=x\succ_{\omega}y-y\prec_{\omega}x,\;
\forall \omega\in \Omega_c,
\end{eqnarray}
\begin{eqnarray}\label{eq:prodrightpreLie}
x\triangleleft_{\omega} y:=x\prec_{\omega}y-y\succ_{\omega}x,\;
\forall \omega\in \Omega_c,
\end{eqnarray}
\end{subequations}
then $(A, \triangleright_{\omega}, \Omega)$ is a left pre-Lie family algebra and 
$(A, \triangleleft_{\omega}, \Omega)$ is a right pre-Lie family algebra.
\end{thm}
\begin{proof}
Let $x,y,z\in A$ and $\alpha, \beta\in \Omega_c$. We have 
\begin{eqnarray*}
(x\triangleright_{\alpha}y)\triangleright_{\alpha\beta}z-
x\triangleright_{\alpha}(y\triangleright{_\beta}z)
&=&
(x\succ_{\alpha}y-y\prec_{\alpha}x)\triangleright_{\alpha\beta}z-
x\triangleright_{\alpha}(y\succ_{\beta}z-z\prec_{\beta} y)\cr 
&=&
(x\succ_{\alpha}y-y\prec_{\alpha}x)\succ_{\alpha\beta}z-
z\prec_{\alpha\beta}(x\succ_{\alpha}y-y\prec_{\alpha}x)
\cr&-&
x\succ_{\alpha}(y\succ_{\beta}z-z\prec_{\beta} y)+
(y\succ_{\beta}z-z\prec_{\beta} y)\prec_{\alpha}x 
\cr&=&
\{
(x\succ_{\alpha}y)\succ_{\alpha\beta}z
-x\succ_{\alpha}(y\succ_{\beta}z)
-(z\prec_{\beta}y)\prec_{\alpha}x\cr
&-&z\prec_{\alpha\beta}(y\prec_{\alpha}x)\}
-(y\prec_{\alpha}x)\succ_{\alpha\beta}z
-z\prec_{\alpha\beta}(x\succ_{\alpha} y)\cr
&+&x\succ_{\alpha}(z\prec_{\beta}y)
+(y\succ_{\beta}z)\prec_{x}
\cr&=&
-(x\prec_{\beta}y)\succ_{\alpha\beta}z-
z\prec_{\alpha\beta}(y\succ_{\beta}x)-
(y\prec_{\alpha}x)\succ_{\alpha\beta}z
\cr&-&
z\prec_{\alpha\beta}(x\succ_{\alpha}y)+
x\succ_{\alpha}(z\prec_{\beta}y)+
(y\succ_{\beta}z)\prec_{\alpha}x\cr
&=&-(x\prec_{\beta}y+y\prec_{\alpha}x)\succ_{\alpha\beta}z
-z\prec_{\alpha\beta}(x\succ_{\alpha}y+y\succ_{\beta}x)\cr 
&+&
(y\succ_{\beta}z)\prec_{\alpha}x+x\succ_{\alpha}(z\prec_{\beta}y)\cr
(x\triangleright_{\alpha}y)\triangleright_{\alpha\beta}z-
x\triangleright_{\alpha}(y\triangleright{_\beta}z)
&=&
(y\triangleright_{\beta}x)\triangleright_{\beta\alpha}z-
y\triangleright_{\beta}(x\triangleright_{\alpha}z).
\end{eqnarray*}
Note that the third equal sign above upwards  
is due to Eq.~\eqref{eq:preanti2} 
while the last equal sign one is 
due to  and Eq.~\eqref{eq:preanti1}.
Therefore, $(A, \triangleright_{\omega}, \Omega_c)$
is a left pre-Lie family algebra.

Besides, we have
\begin{eqnarray*}
(x\triangleleft_{\alpha}y)\triangleleft_{\beta} z
-x\triangleleft_{\alpha\beta}(y\triangleleft_{\beta}z)
&=&
(x\prec_{\alpha}y-y\succ_{\alpha}x)\triangleleft_{\beta}z
-x\triangleleft_{\alpha\beta}(y\prec_{\beta}z-z\succ_{\beta}y)
\cr 
&=&
(x\prec_{\alpha}y-y\succ_{\alpha}x)\prec_{\beta}z
-z\succ_{\beta}(x\prec_{\alpha}y-y\succ_{\alpha}x)
\cr&-&
x\prec_{\alpha\beta}(y\prec_{\beta}z-z\succ_{\beta}y)
+(y\prec_{\beta}z-z\succ_{\beta}y)\succ_{\alpha\beta}x\cr 
&=&
\{(x\prec_{\alpha}y)\prec_{\beta}z
-x\prec_{\alpha\beta}(y\prec_{\beta}z)
+z\succ_{\beta}(y\succ_{\alpha}x)\cr
&-&(z\succ_{\beta}y)\succ_{\alpha\beta}x
\}
-(y\succ_{\alpha}x)\prec_{\beta}z
-\succ_{\beta}(x\prec_{\alpha}y)\cr 
&+&
x\prec_{\alpha\beta}(z\succ_{\beta}y)
+(y\prec_{\beta}z)\succ_{\alpha\beta} x\cr 
&=&
x\prec_{\alpha\beta}(y\succ_{\alpha}z+z\succ_{\beta}y)
+(y\prec{\beta}z+z\prec_{\alpha}y)\succ_{\alpha\beta}x
\cr&-&
\{
(y\succ_{\alpha}x)\prec_{\beta}z+
z\succ_{\beta}(x\prec_{\alpha}y)
\}\cr
&=&
x\prec_{\alpha\beta}(z\succ_{\beta}y+y\succ_{\alpha}z)
+(z\prec_{\alpha}y+y\prec{\beta}z)\succ_{\alpha\beta}x
\cr&-&
\{
(z\succ_{\beta}(x\prec_{\alpha}y+
y\succ_{\alpha}x)\prec_{\beta}z)
\}\cr
(x\triangleleft_{\alpha}y)\triangleleft_{\beta} z
-x\triangleleft_{\alpha\beta}(y\triangleleft_{\beta}z)
&=&
(x\triangleleft_{\beta}z)\triangleleft_{\alpha}y-
x\triangleleft_{\beta\alpha}(z\triangleleft_{\alpha}y).
\end{eqnarray*}
Note that the third equal sign upwards in above 
relations is due to Eq.~\eqref{eq:preanti2} 
while the last equal sign one is 
due to  and Eq.~\eqref{eq:preanti1}.
Therefore, $(A, \triangleleft_{\omega}, \Omega_c)$
is a right pre-Lie family algebra.
\end{proof}


\begin{defi}
 An $\Omega_c$-relative Lie algebra is a vector space $A$ such that for any pair 
 $(\alpha, \beta)\in \Omega_c$ there is a  
 operation $[\cdot,\cdot]_{\alpha, \beta}:A\otimes A \rightarrow A$  satisfying, 
 for all $x,y,z\in A$, and for all $\alpha, \beta, \gamma\in \Omega_c$, 
 \begin{subequations}
 \begin{eqnarray}\label{eq:LieFamilyskew}
 [x,y]_{_{\alpha, \beta}}+[y,x]_{_{\beta, \alpha}}=0, 
 \end{eqnarray}
 \begin{eqnarray}\label{eq:LieFamilyJacobi}
 [[x,y]_{_{\alpha, \beta}}, z]_{_{\alpha\beta, \gamma}}+
 [[y, z]_{_{\beta, \gamma}}, x]_{_{\beta\gamma, \alpha}}+
 [[z, x]_{_{\gamma, \alpha}}, y]_{_{\gamma\alpha, \beta}}
 =0.
 \end{eqnarray}
 \end{subequations}
This    relative Lie algebra will simply denoted by 
$(A, [\cdot, \cdot ]_{_{\omega_1, \omega_2}}, \Omega_c)$.
\end{defi}
\begin{thm}\label{theorelativeantiflexible-Lie}
Let $(A, \prec_{\omega}, \succ_{\omega}, \Omega_c)$
be a pre-anti-flexible family algebra, 
defining for any $x,y\in A$ and for any $\alpha, \beta\in \Omega,$
\begin{eqnarray}\label{eq:commutatorAntiflexible}
[x,y]_{_{\alpha, \beta}}=x\ast_{_{\alpha, \beta}}y-
y\ast_{_{\beta, \alpha}} x
=(x\succ_{\alpha}y+x\prec_{\beta}y)-
(y\succ_{\beta}x+y\prec_{\alpha}x),
\end{eqnarray}
then  $(A, [\cdot, \cdot ]_{_{\alpha, \beta}}, \Omega_c)$ is an 
$\Omega_c$-relative Lie algebra.
\end{thm}
\begin{proof}
Let $x,y,z\in A$. For any $\alpha, \beta, \gamma \in \Omega_c$, 
\begin{itemize}
\item Firstly, we have
\begin{eqnarray*}
[x,y]_{_{\alpha, \beta}}+[y, x]_{_{\beta, \alpha}}=
x\ast_{_{\alpha, \beta}}y-y\ast_{_{\beta, \alpha}} x+
y\ast_{_{\beta, \alpha}} x-x\ast_{_{\alpha, \beta}}y=0.
\end{eqnarray*}
Thus,  Eq.~\eqref{eq:LieFamilyskew} is satisfied.
\item Secondly, we have
\begin{eqnarray*}
[[x,y]_{_{\alpha, \beta}}, z]_{_{\alpha\beta, \gamma}}+
 [[y, z]_{_{\beta, \gamma}}, x]_{_{\beta\gamma, \alpha}}+
 [[z, x]_{_{\gamma, \alpha}}, y]_{_{\gamma\alpha, \beta}}
&=&
 (x\ast_{_{\alpha,\beta}}y)\ast_{_{\alpha\beta, \gamma}}z-
z\ast_{_{\gamma,\alpha\beta}}(x\ast_{_{\alpha,\beta}}y)
\cr &-&
(y\ast_{_{\beta,\alpha}}x)\ast_{_{\alpha\beta,\gamma}}z+
z\ast_{_{\gamma,\alpha\beta}}(y\ast_{_{\beta,\alpha}}x) 
 \cr &+&
(y\ast_{_{\beta,\gamma}}z)\ast_{_{\beta\gamma,\alpha}}x-
x\ast_{_{\alpha,\beta\gamma}}(y\ast_{_{\beta,\gamma}}z)
 \cr &-&
(z\ast_{_{\gamma,\beta}}y)\ast_{_{\beta\gamma,\alpha}}x+
x\ast_{_{\alpha,\beta\gamma}}(z\ast_{_{\gamma,\beta}}y)
 \cr &+&
(z\ast_{_{\gamma,\alpha}}x)\ast_{_{\gamma\alpha,\beta}}y-
y\ast_{_{\beta,\gamma\alpha}}(z\ast_{_{\gamma,\alpha}}x) 
\cr &-&
(x\ast_{_{\alpha,\gamma}}z)\ast_{_{\gamma\alpha,\beta}}y+
y\ast_{_{\beta,\gamma\alpha}}(x\ast_{_{\alpha,\gamma}}z)\cr
&=&
(x,y,z)_{_{\alpha,\beta, \gamma}}
+(y,z,x)_{_{\beta,\gamma,\alpha}}
+(z,x,y)_{_{\gamma, \alpha, \beta}}\cr&-&
(z, y ,z)_{_{\gamma,\beta,\alpha}}
-(x, z, y)_{_{\alpha,\gamma,\beta}}-
(y, x, z)_{_{\beta,\alpha,\gamma}}.
\end{eqnarray*}
The last equal sign above is due to
the cmmutativity of $\Omega_c$.
Thus, according to Theorem~\ref{thm:underlingrelaticeantiflexible} and  Eq.~\eqref{eq:associatorid}, the 
$\Omega_c$-relative Jacobi  identity i.e. Eq.~\eqref{eq:LieFamilyJacobi} is 
satisfied.
\end{itemize}
Therefore,
$(A, [\cdot, \cdot ]_{_{\alpha, \beta}}, \Omega_c)$ is an 
$\Omega_c$-relative Lie algebra.
\end{proof}
\begin{pro}
Let $(A, \prec_{\omega}, \succ_{\omega}, \Omega_c)$
be a pre-anti-flexible family algebra,
the following linear products family,
for all $\alpha, \beta\in \Omega_c$ and for any $x,y\in A,$
\begin{eqnarray}\label{eq:productcomm}
[x,y]_{_{\alpha, \beta}}:= x\triangleright_{\alpha} y
- y\triangleright_{\beta} x,
\end{eqnarray}
where "$\triangleright$" is given by 
Eq.~\eqref{eq:prodleftpreLie}, 
turns $A$ into an $\Omega_c$-relative Lie algebra, which is the 
same $\Omega_c$-relative Lie algebra given in
Theorem~\ref{theorelativeantiflexible-Lie}.
\end{pro}
\begin{proof}
Let $x,y\in A$ and $\alpha, \beta\in \Omega_c$. We have
\begin{eqnarray*}
[x,y]_{_{\alpha, \beta}}
&:=&x\triangleright_{\alpha}y-y\triangleright_{\beta}x
=x\succ_{\alpha}y-y\prec_{\alpha}x
-y\succ_{\beta}x+x\prec_{\beta}y\cr 
&=&(x\succ_{\alpha}y+x\prec_{\beta}y)
-(y\succ_{\beta}x+y\prec_{\alpha}x)
=x\ast_{\alpha, \beta}y-y\ast_{\beta, \alpha}x.
\end{eqnarray*}
According to Theorem~\ref{theorelativeantiflexible-Lie}, 
we deduce that the product given by Eq.~\eqref{eq:productcomm} 
is the same $\Omega_c$-relative Lie algebra built in 
Theorem~\ref{theorelativeantiflexible-Lie}.
\end{proof}
\begin{rmk}
Let us notice that for any $x,y\in A $  and for 
any $\alpha, \beta\in \Omega_c$, 
\begin{eqnarray*}
 x\triangleleft_{\alpha}y-y\triangleleft_{\beta}x&=&
 x\prec_{\alpha}y-y\succ_{\alpha}x
 -y\prec_{\beta}x+x\succ_{\beta}y
 = ( x\prec_{\alpha}y+x\succ_{\beta}y)
 -(y\succ_{\alpha}x+y\prec_{\beta}x)
\cr
 x\triangleleft_{\alpha}y-y\triangleleft_{\beta}x
&=&x\ast_{\beta, \alpha}y-y\ast_{\alpha, \beta}x
=[x, y]_{_{\beta, \alpha}}.
\end{eqnarray*}
Clearly, the $\Omega_c$-relative Lie algebra
underlying pre-anti-family algebra is that 
given by the commutator of Eq.~\eqref{eq:prodrightpreLie}
and that of the commutator of
Eq.~\eqref{eq:underlyinganti}.
\end{rmk}

\begin{thm}\label{thm_anticommutatorantiflexible}
Let $(A, \prec_{\omega}, \succ_{\omega}, \Omega_c)$
be a pre-anti-flexible family algebra such that its
underlying $\Omega_c$-relative 
anti-flexible algebra is 
$(A, \ast_{\omega_1, \omega_2}, \Omega_c)$.
The linear product given by for any 
$x,y\in A$ and for any $\alpha, \beta \in \Omega_c$, 
\begin{eqnarray}
x\circ_{_{\alpha, \beta}} y=x\ast_{\alpha,\beta}y+y\ast_{\beta, \alpha} x, 
\end{eqnarray}
is such that the family algebra $(A, \circ_{\omega_1, \omega_2}, \Omega_c)$ 
satisfying the following relation, 
for any $x,y,z\in A$ and for any  $\alpha, \beta, \gamma\in \Omega_c$,
\begin{eqnarray}
(x,y,z)_{{_\circ}_{_{\alpha, \beta, \gamma}}}=
[y, [x,z]_{_{\alpha, \gamma}}]_{_{\beta, \alpha\gamma}}, 
\end{eqnarray}
where, $(x,y,z)_{{_\circ}_{_{\alpha, \beta, \gamma}}}
=(x\circ_{\alpha, \beta} y)\circ_{\alpha\beta, \gamma}z
-x\circ_{\alpha, \beta\gamma}(y\circ_{\beta, \gamma} z)$ and 
$[\cdot,\cdot]_{_{\alpha, \beta}}$ 
is given by Eq.~\eqref{eq:commutatorAntiflexible}.
\end{thm}
\begin{proof}
Let $x, y,z\in A$, and for all $\alpha, \beta, \gamma\in \Omega_c$, we have
\begin{eqnarray*}
(x,y,z)_{{_\circ}_{_{\alpha, \beta, \gamma}}}&=&
(x\circ_{\alpha, \beta}y)\circ_{\alpha \beta, \gamma}z-
x\circ_{\alpha, \beta\gamma}(y\circ_{\beta, \gamma} z)\cr
&=&(x\ast_{\alpha, \beta}y+
y\ast_{\beta, \alpha}x)\ast_{\alpha \beta, \gamma}z+
z\ast_{\gamma, \alpha \beta}(x\ast_{\alpha, \beta}y+
y\ast_{\beta, \alpha}x)
\cr&-&
x\ast_{\alpha, \beta\gamma}(y\ast_{\beta, \gamma}z
+z\ast_{\gamma, \beta}y)
-(y\ast_{\beta, \gamma}z
+z\ast_{\gamma, \beta}y)\ast_{\beta\gamma, \alpha}x\cr 
&=&
\{(x\ast_{\alpha, \beta}y)\ast_{\alpha\beta, \gamma}z-
x\ast_{\alpha, \beta\gamma}(y\ast_{\beta, \gamma}z)\}-
\{(z\ast_{\gamma, \beta}y)\ast_{\beta\gamma, \alpha}x-
z\ast_{\gamma, \alpha\beta}(y\ast_{\beta, \alpha}x) \}\cr
&+&
(y\ast_{\beta, \alpha}x)\ast_{\alpha\beta, \gamma}z+
z\ast_{\gamma, \alpha\beta}(x\ast_{\alpha, \beta}y)
-x\ast_{\alpha, \beta\gamma}(z\ast_{\gamma, \beta}y)
-(y\ast_{\beta, \gamma}y)\ast_{\beta\gamma, \alpha}x
\cr&=&
(y\ast_{\beta, \alpha}x)\ast_{\alpha\beta, \gamma}z+
z\ast_{\gamma, \alpha\beta}(x\ast_{\alpha, \beta}y)
-x\ast_{\alpha, \beta\gamma}(z\ast_{\gamma, \beta}y)
-(y\ast_{\beta, \gamma}y)\ast_{\beta\gamma, \alpha}x
\cr&=&
y\ast_{\beta, \alpha\gamma}(x\ast_{\alpha, \gamma}z)
+(z\ast_{\gamma, \alpha}x)\ast_{\gamma\alpha, \beta}y
-y\ast_{\beta, \gamma\alpha}(z\ast_{\gamma, \alpha}x)
-(x\ast_{\alpha, \gamma}z)\ast_{\alpha\gamma, \beta}y\cr 
&=&
[y, [x,z]_{_{\alpha, \gamma}}]_{_{\beta, \alpha\gamma}}.
\end{eqnarray*}
Note that the three last equals sign upwards is due to 
Eq.~\eqref{eq:identityantiflexible}.
\end{proof}
\begin{pro}
Let $(A, \ast_{\omega_1, \omega_2}, \Omega_c)$ be and $\Omega_c$-relative 
anti-flexible algebra. Considering the
algebra $(A, \circ_{\omega_1, \omega_2}, \Omega_c)$ given above, 
we have for any $x,y,z\in A$ and for any 
$\alpha, \beta, \gamma \in \Omega_c$, we have
\begin{eqnarray}\label{eq:identitycommutatorntiflexible}
(x,y,z)_{{_\circ}_{_{\alpha, \beta, \gamma}}}+
(z,x,y)_{{_\circ}_{_{\gamma, \alpha, \beta}}}+
(y,z,x)_{{_\circ}_{_{\beta, \gamma, \alpha}}}=0.
\end{eqnarray}
\end{pro}
\begin{proof}
According to Theorem~\ref{theorelativeantiflexible-Lie} and 
Theorem~\ref{thm_anticommutatorantiflexible}, 
Eq.~\eqref{eq:identitycommutatorntiflexible} is satisfied.
\end{proof}

\section{Relative pre-anti-flexible algebras 
 and associated relative algebras}\label{section3}
 Relative pre-anti-flexible algebras structures are 
 introduced and associated relative algebras 
 structures are built. Moreover, 
 relative pre-anti-flexible algebras structures are
 view as a generalization of 
 pre-anti-flexible family algebraic structures and 
 associated consequences are deduct.
\begin{defi}
An $\Omega_c$-relative pre-anti-flexible algebra
consists of a vector space $A$ equipped with
operations 
$\prec_{\alpha, \beta}; 
\succ_{\alpha, \beta}: A\times A\rightarrow A$ for 
each pair $(\alpha, \beta)\in \Omega_c^2$ and satisfying
for any $x,y,z\in A$ and for any 
$\alpha, \beta, \gamma\in \Omega_c$
\begin{subequations}
\begin{eqnarray}\label{eq:relativepreanti1}
(x\succ_{\alpha, \beta}y)\prec_{\alpha\beta, \gamma}z-
x\succ_{\alpha, \beta\gamma}(y\prec_{\beta, \gamma}z)=
(z\succ_{\gamma, \beta}y)\prec_{\gamma\beta, \alpha}x-
z\succ_{\gamma, \beta\alpha}(y\prec_{\beta, \alpha}x), 
\end{eqnarray}
\begin{eqnarray}\label{eq:relativepreanti2}
(x\prec_{\alpha, \beta}y+x\succ_{\alpha, \beta} y)\succ_{\alpha\beta, \gamma}z-
x\succ_{\alpha, \beta\gamma}(y\succ_{\beta, \gamma} z)=\cr 
(z\prec_{\gamma, \beta}y)\prec_{\gamma\beta, \alpha}x-
z\prec_{\gamma, \beta\alpha}(y\prec_{\beta,\alpha}x+
y\succ_{\beta, \alpha}x).
\end{eqnarray}
\end{subequations}
\end{defi}
In the following of this paper, the quadruple 
$(A,\prec_{\omega_1, \omega_2}, \succ_{\omega_1, \omega_2}, \Omega_c)$
 designates an $\Omega_c$-relative 
pre-anti-flexible algebra.
\begin{rmk}
Note that if the LSH and RHS of Eq.~\eqref{eq:relativepreanti1} and 
Eq.~\eqref{eq:relativepreanti2} are zero, then 
$(A,\prec_{\omega_1, \omega_2}, \succ_{\omega_1, \omega_2}, \Omega_c)$
still $\Omega_c$-relative pre-anti-flexible algebra and also
$\Omega_c$-relative dendriform  algebra (\cite{Aguiar}).
Hence, $\Omega_c$-relative pre-anti-flexible algebras are 
$\Omega_c$-relative dendriform algebras.
\end{rmk}
\begin{pro}\label{prop:underreativeantiflexible}
Let $(A,\prec_{\omega_1, \omega_2}, \succ_{\omega_1, \omega_2}, \Omega_c)$
be an $\Omega_c$-relative pre-anti-flexible algebra. The following
linear product given by, for any $x,y\in A$ and for any 
$\alpha, \beta \in \Omega_c$,
\begin{eqnarray}\label{eq:underlyingrelativeantiflexible}
x\circledast_{\alpha, \beta} y=
x\prec_{\alpha, \beta} y+x\succ_{\alpha, \beta} y,
\end{eqnarray}
endows to $A$ an $\Omega_c$-relative anti-flexible algebra.
\end{pro}
\begin{proof}
Let $x,y,z\in A$ and $\alpha, \beta, \gamma \in \Omega_c$. We have
\begin{eqnarray*}
(x,y,z)_{\circledast_{\alpha, \beta, \gamma}}&=&
(x\succ_{\alpha, \beta}y+x\prec_{\alpha, \beta}y)\succ_{\alpha\beta, \gamma}z+
(x\succ_{\alpha, \beta}y+x\prec_{\alpha, \beta}y)\prec_{\alpha\beta, \gamma}z
\cr  
&-&
x\succ_{\alpha, \beta\gamma}(y\succ_{\beta, \gamma}z+y\prec_{\beta, \gamma}z)-
x\prec_{\alpha, \beta\gamma}(y\succ_{\beta, \gamma}z+y\prec_{\beta, \gamma}z)
\cr 
&=&
\{
(x\succ_{\alpha, \beta}y+x\prec_{\alpha, \beta}y)\succ_{\alpha\beta, \gamma}z
-x\succ_{\alpha, \beta\gamma}(y\succ_{\beta, \gamma}z)
\}
\cr&-&
\{
x\prec_{\alpha, \beta\gamma} (y\succ_{\beta,\gamma}z+y\prec_{\beta, \gamma}z)-
(x\prec_{\alpha, \beta}y)\prec_{\alpha\beta, \gamma}z
\}
\cr&+&
\{
(x\succ_{\alpha, \beta}y)\prec_{\alpha\beta, \gamma}z-
x\succ_{\alpha,\beta\gamma}(y\prec_{\beta, \gamma}z)
\}\cr
&=&
\{
(z\prec_{\gamma, \beta}y)\prec_{\gamma\beta, \alpha}x
-z\prec_{\gamma, \beta\alpha}(y\prec_{\beta, \alpha}x+y\succ_{\beta, \alpha}x)
\}\cr
&+&
\{
(z\prec_{\gamma, \beta}y+z\succ_{\gamma, \beta}y)\succ_{\gamma\beta, \alpha}x-
z\succ_{\gamma, \beta\gamma}(y\succ_{\beta, \alpha}x)
\}\cr
&+&
\{
(z\succ_{\gamma, \beta}y)\prec_{\gamma\beta, \alpha}x-
z\succ_{\gamma, \beta\alpha}(y\prec_{\beta, \alpha}x)
\}\cr
&=&
(z\circledast_{\gamma, \beta}y)\circledast_{\gamma\beta, \alpha}x-
z\circledast_{\gamma, \beta\alpha}(y\circledast_{\beta, \alpha}x)=
(z,y,x)_{\circledast_{\gamma, \beta, \alpha}}
\end{eqnarray*}
Note that the third equal sign upwards above is due to 
Eq.~\eqref{eq:relativepreanti1} and 
Eq.~\eqref{eq:relativepreanti2}.
Therefore, $(A, \circledast_{_{\omega_1, \omega_2}}, \Omega_c)$ 
is an $\Omega_c$-relative anti-flexible algebra.
\end{proof}
\begin{defi}
An $\Omega_c$-relative pre-Lie  algebra is a vector space $A$ equipped 
the family of operations
$\ast_{_{\alpha,\beta}}:A\otimes A\rightarrow A$
 for each couple $(\alpha, \beta) \in \Omega_c^2$ 
 such that for any $x,y,z\in A$,
 and for any $\alpha, \beta, \gamma\in \Omega_c$, 
 \begin{eqnarray}\label{eq:associatorpreLieid}
 (x,y,z)_{_{\alpha,\beta, \gamma}}=(y,x,z)_{_{\beta, \alpha, \gamma}},
 \end{eqnarray}
 or equivalently
 \begin{eqnarray}\label{eq:preLieid}
 (x\ast_{_{\alpha, \beta}} y)\ast_{_{\alpha\beta, \gamma}}z-
 x\ast_{_{\alpha, \beta\gamma}}(y\ast_{_{\beta, \gamma}}z)-
 (y\ast_{_{\beta, \alpha}} x)\ast_{_{\beta\alpha, \gamma}}z+
 y\ast_{_{\beta, \alpha\gamma}}(x\ast_{_{\alpha, \gamma}}z)=0.
 \end{eqnarray}
\end{defi}
\begin{thm}
Let $(A,\prec_{\omega_1, \omega_2}, \succ_{\omega_1, \omega_2}, \Omega_c)$
be an $\Omega_c$-relative pre-anti-flexible algebra,
defining for all $\alpha, \beta\in \Omega$ and 
for any $x,y\in A$, 
\begin{eqnarray}\label{eq:underlyingpreLie1}
x \blacktriangleright_{_{\alpha, \beta}}y= 
x\succ_{\alpha, \beta} y- y\prec_{\beta, \alpha} x,
\end{eqnarray}
then 
$(A, \blacktriangleright_{_{\alpha, \beta}}, \Omega_c)$ is an
$\Omega_c$-relative pre-Lie  algebra.  
\end{thm}
\begin{proof}
Let $x,y,z\in A$ and let $\alpha, \beta, \gamma\in \Omega_c$. 
We have 
\begin{eqnarray*}
(x,y,z)_{\blacktriangleright_{\alpha, \beta, \gamma}}&=&
(x\blacktriangleright_{_{\alpha, \beta}} y)
\blacktriangleright_{_{\alpha\beta, \gamma}}z
-x\blacktriangleright_{_{\alpha, \beta\gamma}}(y\blacktriangleright_{_{\beta, \gamma}}z)\cr
&=&
(x\succ_{\alpha, \beta}y-y\prec_{\beta, \alpha}x)\succ_{\alpha\beta, \gamma}z-
z\prec_{\gamma, \alpha\beta}(x\succ_{\alpha, \beta}y-y\prec_{\beta, \alpha}x)
\cr&-&
x\succ_{\alpha, \beta\gamma}(y\succ_{\beta, \gamma}z-z\prec_{\gamma, \beta}y)+
(y\succ_{\beta, \gamma}z-z\prec_{\gamma, \beta}y)\prec_{\beta\gamma, \alpha}x
\cr
&=&
\{
(x\succ_{\alpha, \beta}y)\succ_{\alpha\beta, \gamma}z-
x\succ_{\alpha, \beta\gamma}(y\succ_{\beta, \gamma}z)
\}
\cr
&-&
\{
(z\prec_{\gamma, \beta}y)\prec_{\beta\gamma, \alpha}x-
z\prec_{\gamma, \beta\alpha}(y\prec_{\beta, \alpha}x)
\}
\cr
&+&
\{
(y\succ_{\beta, \gamma}z)\prec_{\beta\gamma, \alpha}x+
x\succ_{\alpha, \beta\gamma}(z\prec_{\gamma, \beta}y)
\}
\cr
&-&
\{
(y\prec_{\beta, \alpha}x)\succ_{\alpha\beta, \gamma}z+
z\prec_{\gamma, \alpha\beta}(x\succ_{\alpha, \beta}y)
\}
\cr
&=&
\{
(x\succ_{\alpha, \gamma}z)\prec_{\alpha\gamma, \beta}y+
y\succ_{\beta, \alpha\gamma}(z\prec_{\gamma, \alpha}x)
\}
\cr
&-&
\{
(x\succ_{\alpha, \beta}y+y\prec_{\beta, \alpha}x)\succ_{\alpha\beta, \gamma}z
\}
\cr
&-&
\{
z\prec_{\gamma, \alpha\gamma}(x\succ_{\alpha, \beta}y+y\succ_{\beta, \alpha}x)
\}
=
(y,x,z)_{\blacktriangleright_{\beta, \alpha,  \gamma}}
\end{eqnarray*}
Note that  the second equal sign upwards in the above successive relations 
is due to Eq.~\eqref{eq:relativepreanti1} and Eq.~\eqref{eq:relativepreanti2}.
Therefore, 
$(A, \blacktriangleright_{_{\alpha, \beta}}, \Omega_c)$ is an
$\Omega_c$-relative pre-Lie  algebra. 
\end{proof}
\begin{thm}
Let $A$ be a $\mathbf{k}$ vector space,
$\Omega_c$ is a commutative associative semi-group.
Consider the linear products 
$\prec, \succ: A\otimes\mathbf{k}\Omega_c\times A\otimes\mathbf{k}\Omega_c
\rightarrow A\otimes\mathbf{k}\Omega_c$ defined on 
$A\otimes\mathbf{k}\Omega_c$. The triple 
$(A\otimes\mathbf{k}\Omega_c, \prec, \succ)$ 
is a pre-anti-flexible algebra if and only if 
$(A, \prec_{\omega_1, \omega_2}, \succ_{\omega_1, \omega_2})$ is a relative
pre-anti-flexible  algebra, where for any 
$x,y,\in A$ and for any $\alpha, \beta\in \Omega_c$, 
\begin{subequations}
\begin{eqnarray}\label{eq:product-1}
(x\otimes \alpha)\prec (y\otimes \beta):=
(x\prec_{\alpha, \beta} y)\otimes \alpha\beta, 
\end{eqnarray}
\begin{eqnarray}\label{eq:product-2}
(x\otimes \alpha)\succ (y\otimes \beta):=
(x\succ_{\alpha, \beta}y)\otimes \alpha\beta.
\end{eqnarray}
\end{subequations}
\end{thm}
\begin{proof}
Let $x,y,z\in A$ and $\alpha, \beta, \gamma \in \Omega_c$.
We have
\begin{eqnarray*}
((x\otimes \alpha)\succ (y\otimes \beta))\prec (z\otimes \gamma)-
(x\otimes\alpha)\succ ((y\otimes \beta)\prec (z\otimes \gamma))\cr=
(
(x\succ_{\alpha, \beta}y)\prec_{\alpha\beta, \gamma}z
-x\succ_{\alpha, \beta\gamma}(y\prec_{\beta, \gamma}z)
)\otimes\alpha\beta\gamma,
\end{eqnarray*}
\begin{eqnarray*}
((x\otimes \alpha)\succ(y\otimes \beta)+
(x\otimes \alpha)\prec(y\otimes \beta))\succ (z\otimes \gamma)-
(x\otimes\alpha)\succ ((y\otimes \beta)\succ(z\otimes\gamma))\cr
=
(
(x\succ_{\alpha, \beta}y
+x\prec_{\alpha, \beta}y)\succ_{\alpha\beta, \gamma}z
-x\succ_{\alpha, \beta\gamma}(y\succ_{\beta, \gamma}z)
)\otimes\alpha\beta\gamma.
\end{eqnarray*}
Using the commutativity of $\Omega_c$, we have 
\begin{eqnarray*}
((z\otimes \gamma)\succ (y\otimes \beta))\prec (x\otimes \alpha)-
(z\otimes\gamma)\succ ((y\otimes \beta)\prec (x\otimes \alpha))\cr=
 (
 (z\succ_{\gamma, \beta}y)\prec_{\gamma\beta, \alpha}x
 -z\succ_{\gamma, \beta\alpha}(y\prec_{\beta, \alpha}x)
 )\otimes\alpha\beta\gamma,
\end{eqnarray*}
\begin{eqnarray*}
((z\otimes  \gamma)\prec(y\otimes\beta))\prec (x\otimes\alpha)-
(z\otimes \gamma)\prec ((y\otimes \beta)\succ(x\otimes\alpha)
+(y\otimes\beta)\prec(x\otimes\alpha) )\cr=
 (
 (z\prec_{\gamma, \beta}y)\prec_{\gamma\beta, \alpha}x
 -z\prec_{\gamma, \beta\alpha}
 (y\succ_{\beta, \alpha}x+y\prec_{\beta,\alpha}x)
 )\otimes\alpha\beta\gamma.
\end{eqnarray*}
Thus if $(A\otimes\mathbf{k}\Omega_c, \prec, \succ)$ is 
a pre-anti-flexible algebra then
$(A, \prec_{\omega_1, \omega_2}, \succ_{\omega_1, \omega_2})$ 
is an $\Omega_c$-relative pre-anti-flexible  algebra, and 
conversely, if $(A, \prec_{\omega_1, \omega_2}, \succ_{\omega_1, \omega_2})$ 
is an $\Omega_c$-relative pre-anti-flexible  algebra wiht 
product given by Eq.~\eqref{eq:product-1} and 
Eq.~\eqref{eq:product-2}, then 
$(A\otimes\mathbf{k}\Omega_c, \prec, \succ)$ is 
a pre-anti-flexible algebra.
\end{proof}
\begin{defi}
An $\Omega_c$-relative right pre-Lie  algebra 
(or $\Omega_c$-relative right symmetric algebra)
is a vector space $A$ equipped 
the family of operations
"$\cdot_{_{\alpha,\beta}}:A\otimes A\rightarrow A$"
 for each couple $(\alpha, \beta) \in \Omega_c^2$ 
 such that for any $x,y,z\in A$,
 and for any $\alpha, \beta, \gamma\in \Omega_c$, 
 \begin{eqnarray}\label{eq:associatorrightLieid}
 (x,y,z)_{_{\alpha,\beta, \gamma}}=
 (x,z,y)_{_{\alpha, \gamma, \beta}},
 \end{eqnarray}
 or equivalently
 \begin{eqnarray}\label{eq:rightLieid}
 (x\cdot_{_{\alpha, \beta}} y)\cdot_{_{\alpha\beta, \gamma}}z-
 x\cdot_{_{\alpha, \beta\gamma}}(y\cdot_{_{\beta, \gamma}}z)-
 (x\cdot_{_{\alpha, \gamma}} z)\cdot_{_{\alpha\gamma, \beta}}y+
 x\cdot_{_{\alpha, \gamma\beta}}(z\cdot_{_{\gamma, \beta}}y)=0.
 \end{eqnarray}
\end{defi}
\begin{pro}
Let $(A,\prec_{\omega_1, \omega_2}, \succ_{\omega_1, \omega_2}, \Omega_c)$
be an $\Omega_c$-relative pre-anti-flexible algebra,
defining for all $\alpha, \beta\in \Omega$ and 
for any $x,y\in A$,
\begin{eqnarray}\label{eq:underlyingpreLie2}
x \blacktriangleleft_{_{\alpha, \beta}}y= 
x\prec_{\alpha, \beta} y- y\succ_{\beta, \alpha} x,
\end{eqnarray}
and 
$(A, \blacktriangleleft_{_{\alpha, \beta}}, \Omega_c)$ is
an $\Omega_c$-relative pre-Lie  algebra.
\end{pro}
\begin{proof}
By straightforward calculation.
\end{proof}
\begin{thm}\label{thm_underLie}
Let $(A,\prec_{\omega_1, \omega_2}, \succ_{\omega_1, \omega_2}, \Omega_c)$
be an $\Omega_c$-relative pre-anti-flexible algebra.
There is an $\Omega_c$-relative Lie algebra structure
underling the $\Omega_c$-relative anti-flexible algebra 
$(A, \circledast_{_{\omega_1, \omega_2}}, \Omega_c)$, 
derived in  Proposition~\ref{prop:underreativeantiflexible}, given by, 
for any $x,y\in A$ and for any $\alpha, \beta\in \Omega_c$, 
\begin{eqnarray}
[x,y]_{_{\alpha, \beta}}:=
x\circledast_{_{\alpha, \beta}}y-
y\circledast_{_{\beta, \alpha}}x=
(x\succ_{\alpha, \beta}y+x\prec_{\alpha, \beta}y)-
(y\succ_{\beta,\alpha}x+y\prec_{\beta, \alpha}x).
\end{eqnarray}
\end{thm}
\begin{proof}
Let $x,y,z\in A$ and $\alpha, \beta, \gamma\in \Omega_c$. We have
\begin{eqnarray*}
[x,y]_{_{\alpha, \beta}}+
[y, x]_{_{\beta, \alpha}}=
x\circledast_{_{\alpha, \beta}}y-
y\circledast_{_{\beta, \alpha}}x
+y\circledast_{_{\beta, \alpha}}x-
x\circledast_{_{\alpha, \beta}}y=0.
\end{eqnarray*}
In addition, we have
\begin{eqnarray*}
[[x,y]_{_{\alpha, \beta}}, z]_{_{\alpha\beta, \gamma}}+
 [[y, z]_{_{\beta, \gamma}}, x]_{_{\beta\gamma, \alpha}}+
 [[z, x]_{_{\gamma, \alpha}}, y]_{_{\gamma\alpha, \beta}}
&=&
(x,y,z)_{\circledast_{\alpha,\beta, \gamma}}
+(y,z,x)_{\circledast_{\beta,\gamma,\alpha}}
+(z,x,y)_{\circledast_{\gamma, \alpha, \beta}}\cr&-&
(z, y ,z)_{\circledast_{\gamma,\beta,\alpha}}
-(x, z, y)_{\circledast_{\alpha,\gamma,\beta}}-
(y, x, z)_{\circledast_{\beta,\alpha,\gamma}}.
\end{eqnarray*}
According to Proposition~\ref{prop:underreativeantiflexible}, 
we deduce the $\Omega_c$-relative Jacobi identity.
Therefore, $A$ contains an $\Omega_c$-relative Lie algebra structure.
\end{proof}
\begin{thm}\label{thm_underalgebra}
Let $(A,\prec_{\omega_1, \omega_2}, \succ_{\omega_1, \omega_2}, \Omega_c)$
be an $\Omega_c$-relative pre-anti-flexible algebra such that its 
underlying $\Omega_c$-relative anti-flexible algebra is
$(A, \circledast_{_{\omega_1, \omega_2}}, \Omega_c)$.
The linear product given by for any $x,y\in A$ and 
for any $\alpha, \beta\in \Omega_c$, 
\begin{eqnarray}
x\circledcirc_{_{\alpha, \beta}} y=
x\circledast_{_{\alpha, \beta}}y+ 
y\circledast_{_{\beta, \alpha}}x,
\end{eqnarray}
is such that $(A, \circledcirc_{_{\omega_1, \omega_2}}, \Omega_c)$
satisfying the following relation, for any 
$x,y,z\in A$ and for any $\alpha, \beta, \gamma\in \Omega_c$, 
\begin{eqnarray}
(x,y,z)_{\circledcirc_{\alpha, \beta, \gamma}}=
[y, [x,z]_{_{\alpha, \gamma}}]_{_{\beta, \alpha\gamma}}, 
\end{eqnarray}
where
$(x,y,z)_{\circledcirc_{\alpha, \beta, \gamma}}=
(x\circledcirc_{_{\alpha, \beta}}y)\circledcirc_{_{\alpha\beta, \gamma}}z-
x\circledcirc_{_{\alpha, \beta\gamma}}(y\circledcirc_{_{\beta, \gamma}}z)$
and $[x,y]_{_{\alpha, \beta}}=
x\circledast_{_{\alpha, \beta}}y-
y\circledast_{_{\beta, \alpha}}x$.
\end{thm}
\begin{proof}
Let $x,y,z\in A$ and $\alpha, \beta, \gamma\in \Omega_c$. We have
\begin{eqnarray*}
(x,y,z)_{\circledcirc_{\alpha, \beta, \gamma}}&=&
(x\circledcirc_{_{\alpha, \beta}}y)\circledcirc_{_{\alpha\beta, \gamma}}z-
x\circledcirc_{_{\alpha, \beta\gamma}}(y\circledcirc_{_{\beta, \gamma}}z)\cr 
&=&
(x\circledast_{_{\alpha, \beta}}y)\circledast_{_{\alpha\beta, \gamma}}z+
(y\circledast_{_{\beta, \alpha}}x)\circledast_{_{\beta\alpha, \gamma}}z
+z\circledast_{_{\gamma, \alpha\beta}}(x\circledast_{_{\alpha, \beta}}y)
+z\circledast_{_{\gamma, \alpha\beta}}(y\circledast_{_{\beta, \alpha}}x)
\cr&-&
x\circledast_{_{\alpha, \beta\gamma}}(y\circledast_{_{\beta, \gamma}}z)-
x\circledast_{_{\alpha, \beta\gamma}}(z\circledast_{_{\gamma, \beta}}y)-
(y\circledast_{_{\beta, \gamma}}z)\circledast_{_{\beta\gamma, \alpha}}x-
(z\circledast_{_{\gamma, \beta}}y)\circledast_{_{\beta\gamma, \alpha}}x
\cr
&=&
(y\circledast_{_{\beta, \alpha}}x)\circledast_{_{\beta\alpha, \gamma}}z
+z\circledast_{_{\gamma, \alpha\beta}}(x\circledast_{_{\alpha, \beta}}y)-
x\circledast_{_{\alpha, \beta\gamma}}(z\circledast_{_{\gamma, \beta}}y)-
(y\circledast_{_{\beta, \gamma}}z)\circledast_{_{\beta\gamma, \alpha}}x
\cr
&=&
y\circledast_{_{\beta, \alpha\gamma}}(x\circledast_{_{\alpha, \gamma}}z)
+(z\circledast_{_{\gamma, \alpha}}x)\circledast_{_{\gamma\alpha, \beta}}y-
(x\circledast_{_{\alpha, \gamma}}z)\circledast_{_{\alpha\gamma, \beta}}y-
y\circledast_{_{\beta, \gamma\alpha}}(z\circledast_{_{\gamma, \alpha}}x)
\cr
&=&[y, [x,z]_{_{\alpha, \gamma}}]_{_{\beta, \alpha\gamma}}.
\end{eqnarray*}
Note that second and third equal sign upward  are due to 
Proposition~\ref{prop:underreativeantiflexible}.
\end{proof}
\begin{pro}
Let $(A,\prec_{\omega_1, \omega_2}, \succ_{\omega_1, \omega_2}, \Omega_c)$
be an $\Omega_c$-relative pre-anti-flexible algebra.
Consider the algebra
$(A, \circledcirc_{_{\omega_1, \omega_2}}, \Omega_c)$ defined above. We have for any 
$x,y,z\in A$ and for any $\alpha, \beta, \gamma\in \Omega_c$, 
\begin{eqnarray}
(x,y,z)_{\circledcirc_{\alpha, \beta, \gamma}}+
(z,x,y)_{\circledcirc_{\gamma, \alpha, \beta }}+
(y,z,x)_{\circledcirc_{\beta, \gamma, \alpha}}=0.
\end{eqnarray}
\end{pro}
\begin{proof}
According to Theorem~\ref{thm_underLie} and
Proposition~\ref{prop:underreativeantiflexible} and 
Theorem~\ref{thm_underalgebra}, the above equation is satisfied.
\end{proof}

\begin{pro}
Let $(A,\prec_{\omega_1, \omega_2}, \succ_{\omega_1, \omega_2}, \Omega_c)$
be an $\Omega_c$-relative pre-anti-flexible algebra.
Suppose that, for any $\omega_1, \omega_2\in \Omega_c$,  the operations 
$\prec_{\omega_1, \omega_2}$ are independent of $\omega_1$ and the 
operations $\succ_{\omega_1, \omega_2} $ are independent of 
$\omega_2$. Then $A$ possesses:
\begin{enumerate}
\item 
a pre-anti-flexible family 
algebra structure and conversely, if 
$A$  possesses a pre-anti-flexible family 
algebra structure, then it is an 
$\Omega_c$-relative pre-anti-flexible algebra such that 
$\prec_{\omega_1, \omega_2}$ are independent of $\omega_1$ and the 
operations $\succ_{\omega_1, \omega_2} $ are independent of 
$\omega_2$. 
\item 
a left pre-Lie family 
algebra structure and conversely, if 
$A$  possesses a left pre-Lie family 
algebra structure, then it is an 
$\Omega_c$-relative  pre-Lie algebra such that 
$\prec_{\omega_1, \omega_2}$ are independent of $\omega_1$ and the 
operations $\succ_{\omega_1, \omega_2} $ are independent of 
$\omega_2$. 
\item 
a right pre-Lie family 
algebra structure and conversely, if 
$A$  possesses a right pre-Lie family 
algebra structure, then it is an 
$\Omega_c$-relative right-symmetric algebra such that 
$\prec_{\omega_1, \omega_2}$ are independent of $\omega_1$ and the 
operations $\succ_{\omega_1, \omega_2}$ are independent of 
$\omega_2$. 
\item an $\Omega_c$-relative Lie algebra structure.
\end{enumerate}
\end{pro}
\begin{proof}
Under divers assumptions,  Eq.~\eqref{eq:preanti1} 
is expressed by Eq.~\eqref{eq:relativepreanti1}
and 
Eq.~\eqref{eq:preanti2} 
is translated by Eq.~\eqref{eq:relativepreanti2}.
\end{proof}
\section{Rota-Baxter 
operators and relative anti-flexible algebras}\label{section4}
This section deals with the using of the 
Rota-Baxter operators underlying relative 
anti-flexible and Lie algebras
and a generalization of  Rota-Baxter 
operator defined on
relative anti-flexible algebra to built  
relative pre-anti-flexible algebra.
It is proved that a Rota-Baxter operator
define underlying 
relative Lie algebra subjoined to 
relative anti-flexible algebra
induces a relative pre-anti-flexible algebra under
some constraints.
\begin{defi}
Let $(A, \cdot_{\omega_1, \omega_2}, \Omega_c)$ be an 
$\Omega_c$-relative anti-flexible algebra.
A Rota-Baxter operator on $A$ is a family of linear operators
$R_{{B}_{\alpha}}:A\rightarrow A$, for any $\alpha \in \Omega_c$ 
which satisfying the following relation, for any 
$x,y\in A$ and any $\alpha, \beta\in \Omega_c$,
\begin{eqnarray}
R_{{B}_{\alpha}}(x)\cdot_{\alpha, \beta}R_{{B}_{\beta}}(y)=
R_{{B}_{\alpha\beta}}(R_{{B}_{\alpha}}(x)\cdot_{\alpha, \beta}y+
x\cdot_{\alpha, \beta}R_{{B}_{\beta}}(y)).
\end{eqnarray}
\end{defi}
\begin{defi}
Let $(A, \cdot_{\omega_1, \omega_2}, \Omega_c)$ 
be an $\Omega_c$-relative anti-flexible algebra.
A generalized Rota-Baxter linear operator on $A$ is 
a family of $G_{{RB}_{\alpha}}:A\rightarrow A$ 
for any $\alpha\in \Omega_c$ such that 
for any $\alpha, \beta, \gamma\in \Omega_c$ and any 
$x, y, z\in A$, we have 
\begin{eqnarray}\label{eq:generalizedRB}
&&(G_{{RB}_{\alpha\beta}}(G_{{RB}_{\alpha}}(x)\cdot_{\alpha, \beta} y
+x\cdot_{\alpha, \beta}G_{{RB}_{\beta}}(y))
-G_{{RB}_{\alpha}}(x)\cdot_{\alpha, \beta}G_{{RB}_{\beta}}(y)
)\cdot_{\alpha\beta, \gamma}z\cr&&+ 
z\cdot_{\gamma,\beta\alpha}
(G_{{RB}_{\beta}}(y)\cdot_{\beta, \alpha}G_{{RB}_{\alpha}}(x)
-G_{{RB}_{\beta\alpha}}(G_{{RB}_{\beta}}(y)\cdot_{\beta, \alpha}x
+y\cdot_{\beta,\alpha}G_{{RB}_{\alpha}}(x)))=0.
\end{eqnarray}
\end{defi}
\begin{pro}\label{prop:GRB-pre-anti-flexible}
Let $G_{{RB}_{\alpha}}: A\rightarrow A$
be a generalized Rota-Baxter linear maps defined 
on an $\Omega_c$-relative anti-flexible algebra 
$(A, \cdot_{\omega_1, \omega_2}, \Omega_c)$.
Defining for any $x,y\in A$ 
and any $\alpha, \beta\in \Omega_c$,
\begin{eqnarray}\label{eq:relativePAF_antiflexible}
x\prec_{\alpha, \beta} y
:=x\cdot_{\alpha, \beta} G_{{RB}_{\beta}}(y),\;\;
x\succ_{\alpha, \beta} y:=G_{{RB}_{\alpha}}(x)\cdot_{\alpha, \beta}y, 
\end{eqnarray}
then, $(A, \prec_{\alpha, \beta}, \succ_{\alpha, \beta}, \Omega_c)$
is turns 
to an $\Omega_c$-relative pre-anti-flexible algebra.
The converse is true.
\end{pro}
\begin{proof}
For any $x,y,z\in A$ and for any $\alpha, \beta, \gamma\in \Omega_c$, we have
\begin{eqnarray*}
(x\succ_{\alpha, \beta}y)\prec_{\alpha\beta, \gamma}z
-x\succ_{\alpha, \beta\gamma}(y\prec_{\alpha, \beta}z)&=&
(G_{{RB}_{\alpha}}(x), y, G_{{RB}_{\gamma}}(z))_{_{\alpha, \beta, \gamma}},\\
(z\succ_{\gamma, \beta}y)\prec_{\gamma\beta, \alpha}x
-z\succ_{\gamma, \beta\alpha}(y\prec_{\beta,  \alpha}x)&=&
(G_{{RB}_{\gamma}}(z), y, G_{{RB}_{\alpha}}(x))_{{\gamma, \beta, \alpha}},\\
(x\prec_{\alpha, \beta}y+x\succ_{\alpha, \beta}y)\succ_{\alpha\beta, \gamma}z-
x\succ_{\alpha, \beta\gamma}(y\succ_{\beta, \gamma}z)&=&
(z, G_{{RB}_{\beta}}(y), G_{{RB}_{\alpha}}(x))_{{\gamma, \beta, \alpha}}+\cr
&&
(G_{{RB}_{\alpha\beta}}(G_{{RB}_{\alpha}}(x)\cdot_{\alpha, \beta}y+
x\cdot_{\alpha, \beta} G_{{RB}_{\beta}}(y))\cr&&-
(G_{{RB}_{\alpha}}(x)\cdot_{\alpha,
\beta}G_{{RB}_{\beta}}(y)))\cdot_{\alpha\beta, \gamma}z,
\\
(z\prec_{\gamma, \beta} y)\prec_{\gamma\beta,\alpha}x
-z\prec_{\gamma, \beta\alpha}
(y\succ_{\beta, \alpha} x+y\prec_{\beta, \alpha}x)&=&
(G_{{RB}_{\alpha}}(x), G_{{RB}_{\beta}}(y), z)_{_{\alpha,
\beta, \gamma}}+\cr
&&z\cdot_{\gamma, \beta\alpha}
(G_{{RB}_{\beta}}(y)\cdot_{ \beta, \alpha}G_{{RB}_{\alpha}}(x)-\cr
&&
G_{{RB}_{\beta\alpha}}(G_{{RB}_{\beta}}(y)\cdot_{\beta, \alpha}x
+y\cdot_{\beta, \alpha}G_{{RB}_{\alpha}}(x))).
\end{eqnarray*}
Therefore, Eq.\eqref{eq:relativePAF_antiflexible} turns 
$A$ into an $\Omega_c$-relative pre-anti-flexible algebra
if and only if $G_{{RB}_{\alpha}}$, for any $\alpha\in A$,
satisfying Eq.~\eqref{eq:generalizedRB}.
\end{proof}
\begin{cor}
Any Rota-Baxter operator on an 
$\Omega_c$-relative anti-flexible algebra 
induces an $\Omega_c$-relative 
pre-anti-flexible algebra.
\end{cor}
In the sequel of this section, we suppose that 
the $\Omega_c$-relative anti-flexible algebra
$(A, \cdot_{\omega_1, \omega_2}, \Omega_c)$ is such that 
the elements 
\begin{eqnarray}\label{eq:element_commute}
\varphi_{_\alpha}(x)\cdot_{_{\alpha, \beta}} \varphi_{_\beta}(y)-
\varphi_{_{\alpha\beta}}(x\cdot_{_{\alpha, \beta}}\varphi_{_\beta}(y)+
\varphi_{_\alpha}(x)\cdot_{_{\alpha, \beta}}y), \;\;
\forall x,y\in A,\; 
\forall \alpha, \beta \in \Omega_c,
\end{eqnarray}
commute with 
any other elements in $A$, where 
$\varphi_{_\alpha}: A\rightarrow A$ is a linear family maps on $A$ for 
$\alpha\in \Omega_c$, that is for any $x,y,z\in A$ and for any 
$\alpha, \beta, \gamma\in \Omega_c$,
we have 
\begin{eqnarray}\label{eq:comu-RB}
z\cdot_{_{\gamma, \beta\alpha}}
(\varphi_{_\alpha}(x)\cdot_{_{\alpha, \beta}} \varphi_{_\beta}(y)-
\varphi_{_{\alpha\beta}}(x\cdot_{_{\alpha, \beta}}\varphi_{_\beta}(y)+
\varphi_{_\alpha}(x)\cdot_{_{\alpha, \beta}}y))=\cr
(\varphi_{_\alpha}(x)\cdot_{_{\alpha, \beta}} \varphi_{_\beta}(y)-
\varphi_{_{\alpha\beta}}(x\cdot_{_{\alpha, \beta}}\varphi_{_\beta}(y)+
\varphi_{_\alpha}(x)\cdot_{_{\alpha, \beta}}y))
\cdot_{_{\alpha\beta, \gamma}}z.
\end{eqnarray}.
\begin{defi}
A Rota-Baxter operator on  a relative Lie algebra
$(A, [\cdot, \cdot]_{_{\omega_1, \omega_2}}, \Omega_c)$
is a family of linear operators $R_{B_{\alpha}}: A\rightarrow A$
for any $\alpha\in \Omega_c$ and satisfying 
for any $x,y\in A$, and for any $\alpha, \beta\in \Omega_c$, 
\begin{eqnarray}\label{eq:RB_Lie}
[R_{{B}_{\alpha}}(x), R_{{B}_{\beta}}(y)]_{_{\alpha, \beta}}=
R_{{B}_{\alpha\beta}}([x, R_{{B}_{\beta}}(y)]_{_{\alpha, \beta}}
+[R_{{B}_{\alpha}}(x), y]_{_{\alpha, \beta}}). 
\end{eqnarray}
\end{defi}
\begin{pro}
Let $(A, \cdot_{\omega_1, \omega_2}, \Omega_c)$ be a relative 
anti-flexible algebra equipped with a linear maps family 
$\varphi_{_\alpha}:A\rightarrow A$, 
for any $\alpha\in \Omega_c$, in which the  elements 
as the form given in Eq.~\eqref{eq:element_commute}
commute with any other element in $A$.
The linear products family given by, 
for any $x,y\in A$ and any $\alpha, \beta\in \Omega_c$, 
\begin{eqnarray}
x\prec_{\alpha, \beta} y=x\cdot_{_{\alpha, \beta}} 
\varphi_{_\beta}(y);\quad
x\succ_{\alpha, \beta} y=
\varphi_{_\alpha}(x)\cdot_{_{\alpha, \beta}} y
\end{eqnarray}
define an $\Omega_c$-relative pre-anti-flexible 
structures if and only if the linear maps 
family $\varphi_{\alpha}$ is a Rota-Baxter operator
on the underlying relative Lie algebra of the relative 
anti-flexible algebra $(A, \cdot_{\omega_1, \omega_2}, \Omega_c)$.
\end{pro}
\begin{proof}
According to Proposition~\ref{prop:GRB-pre-anti-flexible}, 
the linear maps family $\varphi_{\alpha}$ satisfying 
Eq~\eqref{eq:generalizedRB} and 
Eq.~\eqref{eq:comu-RB} we then have  for 
any $x,y,z\in A$ and any $\alpha, \beta, \gamma\in \Omega_c$,
\begin{eqnarray}
z\cdot_{_{\gamma, \alpha\beta}}
([\varphi_{_{\alpha}}(x), \varphi_{_{\beta}}(y)]_{_{\alpha, \beta}}
-\varphi_{_{\alpha\beta}}([x, \varphi_{_{\beta}}(x)]_{_{\alpha, \beta}}+
[\varphi_{_{\alpha}}(x), y]_{_{\alpha, \beta}}))=0.
\end{eqnarray}
Hence $(A, \cdot_{\omega_1, \omega_2}, \Omega_c)$ is 
a trivial relative algebra (i.e for any $x,y\in A$ and any 
$\alpha, \beta\in \Omega_c$, $x\cdot_{_{\alpha, \beta}}y=0$) or 
$\varphi_{_{\alpha}}$ satisfying Eq.~\eqref{eq:RB_Lie}.
Notice that is the case that 
$(A, \cdot_{\omega_1, \omega_2}, \Omega_c)$ is 
a trivial relative algebra, $\varphi_{_{\alpha}}$ 
satisfy Eq.~\eqref{eq:RB_Lie}. Therefore, 
$\varphi_\alpha$ 
is a Rota-Baxter operator
on the underlying relative Lie algebra of the relative 
anti-flexible algebra $(A, \cdot_{\omega_1, \omega_2}, \Omega_c)$.
\end{proof}

\bigskip
\noindent
{\bf Acknowledgments.}   
This work is supported by the LPMC of Nankai University
and Nankai ZhiDe Foundation.

\end{document}